\documentclass[reqno, 11pt]{amsart}
\usepackage{amsmath,amssymb,ascmac,amscd,amsthm,bm,tikz}
\usepackage[all]{xy}
\usepackage{enumitem}
\numberwithin{equation}{section}
\usetikzlibrary[3d]

\usepackage[nobysame,alphabetic,initials,msc-links,lite]{amsrefs}

\DefineSimpleKey{bib}{how}

\renewcommand{\PrintDOI}[1]{%
  \href{http://dx.doi.org/#1}{{\tt DOI:#1}}%
}

\renewcommand{\eprint}[1]{#1}

\BibSpec{book}{%
    +{}  {\PrintPrimary}                {transition}
    +{.} { \PrintDate}                  {date}
    +{.} { \textit}                     {title}
    +{.} { }                            {part}
    +{:} { \textit}                     {subtitle}
    +{,} { \PrintEdition}               {edition}
    +{}  { \PrintEditorsB}              {editor}
    +{,} { \PrintTranslatorsC}          {translator}
    +{,} { \PrintContributions}         {contribution}
    +{,} { }                            {series}
    +{,} { \voltext}                    {volume}
    +{,} { }                            {publisher}
    +{,} { }                            {organization}
    +{,} { }                            {address}
    +{,} { }                            {status}
    +{,} { \PrintDOI}                   {doi}
    +{,} { \PrintISBNs}                 {isbn}
    +{}  { \parenthesize}               {language}
    +{}  { \PrintTranslation}           {translation}
    +{;} { \PrintReprint}               {reprint}
    +{.} { }                            {note}
    +{.} {}                             {transition}
}

\BibSpec{article}{%
    +{}  {\PrintAuthors}                {author}
    +{,} { \textit}                     {title}
    +{.} { }                            {part}
    +{:} { \textit}                     {subtitle}
    +{,} { \PrintContributions}         {contribution}
    +{.} { \PrintPartials}              {partial}
    +{,} { }                            {journal}
    +{}  { \textbf}                     {volume}
    +{}  { \PrintDatePV}                {date}
    +{,} { \issuetext}                  {number}
    +{,} { \eprintpages}                {pages}
    +{,} { }                            {status}
    +{,} { \PrintDOI}                   {doi}
    +{,} { \eprint}        {eprint}
    +{}  { \parenthesize}               {language}
    +{}  { \PrintTranslation}           {translation}
    +{;} { \PrintReprint}               {reprint}
    +{.} { }                            {note}
    +{.} {}                             {transition}
}
\BibSpec{collection.article}{%
    +{}  {\PrintAuthors}                {author}
    +{,} { \textit}                     {title}
    +{.} { }                            {part}
    +{:} { \textit}                     {subtitle}
    +{,} { \PrintContributions}         {contribution}
    +{,} { \PrintConference}            {conference}
    +{}  {\PrintBook}                   {book}
    +{,} { }                            {booktitle}
    +{,} { \PrintDateB}                 {date}
    +{,} { pp.~}                        {pages}
    +{,} { }                            {publisher}
    +{,} { }                            {organization}
    +{,} { }                            {address}
    +{,} { }                            {status}
    +{,} { \PrintDOI}                   {doi}
    +{,} { \eprint}        {eprint}
    +{}  { \parenthesize}               {language}
    +{}  { \PrintTranslation}           {translation}
    +{;} { \PrintReprint}               {reprint}
    +{.} { }                            {note}
    +{.} {}                             {transition}
}

\BibSpec{misc}{%
  +{}{\PrintAuthors}  {author}
  +{,}{ \textit}      {title}
  +{.}{ }             {how}
  +{}{ \parenthesize} {date}
  +{,} { available at \eprint}        {eprint}
  +{,}{ available at \url}{url}
  +{,}{ }             {note}
  +{.}{}              {transition}
}

\newcounter{test}
\numberwithin{test}{section}
  
\theoremstyle{plain}
\newtheorem{theorem}{Theorem}[section]
\newtheorem{lemma}[theorem]{Lemma}
\newtheorem{proposition}[theorem]{Proposition}
\newtheorem{corollary}[theorem]{Corollary}
\newtheorem{conjecture}[theorem]{Conjecture}

\theoremstyle{definition}
\newtheorem{definition}[theorem]{Definition}

\newcommand{\relmid}{\mathrel{}\middle|\mathrel{}}

\newcommand{\bs}{\backslash}
\newcommand{\bR}{{\mathbb{R}}}
\newcommand{\R}{{\mathbb{R}}}
\newcommand{\Qcoh}{{\mathop{\mathrm{Qcoh}}}}
\newcommand{\E}{{\mathcal{E}}}

\newcommand{\Hom}{{\mathrm{Hom}}}

\newcommand{\D}{{\mathrm{D}}}

\newcommand{\dL}{{\mathbb{L}}}

\newcommand{\OX}{{\mathcal{O}}}

\newcommand{\ra}{{\rightarrow}}
\newcommand{\C}{{\mathbb{C}}}
\newcommand{\Sh}{\mathop{\mathrm{Sh}}}
\newcommand{\coh}{\mathop{\mathrm{coh}}}
\newcommand{\rank}{{\mathop{\mathrm{rank}}}}

\newcommand{\qtheta}{{\overline{\Theta'}}}
\newcommand{\ctheta}{{\overline{\Theta}}}
\newcommand{\Xs}{{X_{\Sigma}}}
\newcommand{\Xss}{{X_{\hat{\Sigma}}}}

\newcommand{\Vd}{{}\mathbb{D}}

\newcommand{\eeq}{{\end{equation}}}
\newcommand{\ueb}{\mathrm{b}}
\newcommand{\Z}{\mathbb{Z}}
\newcommand{\torus}{M_\mathbb{R}/M}
\newcommand{\Ls}{\overline{\Lambda_\Sigma}}
\newcommand{\Lss}{\overline{\Lambda_{\hat{\Sigma}}}}
\newcommand{\HSigma}{\hat{\Sigma}}
\newcommand{\coker}{\mathrm{Coker}}
\newcommand{\Cech}{\mathrm{\check{C}ech}}
\newcommand{\ms}{\mathrm{SS}}
\newcommand{\NC}{{\mathrm{NCCC}}}
\newcommand{\Ks}{{\kappa_{\Sigma}}}
\newcommand{\Kss}{{\kappa_{\hat{\Sigma}}}}
\newcommand{\lperp}{{^\perp{\hspace{-0.5mm}}}}
\newcommand{\Relint}{{\mathrm{Relint}}}
\newcommand{\Db}[2]{{\mathrm{D}_{#1}^{\mathrm{b}}\left(#2\right)}}
\newcommand{\hull}[1]{{\left\langle {#1}\right\rangle}}
\newcommand{\set}[1]{{\left\{#1\right\}}}

\newcommand{\Int}{{\mathrm{Int}}}

\newcommand{\pa}[1]{{\left({#1}\right)}}
\newcommand{\eqn}[1]{{\begin{equation}{#1}\end{equation}}}
\newcommand{\RG}[1]{{\R\Gamma\left(#1\right)}}
\newcommand{\RGL}[2]{{\R\Gamma_{#1}\left({#2}\right)}}
\newcommand{\Cl}[1]{{\mathrm{Cl}\left({#1}\right)}}
\newcommand{\tgamma}{{\tilde{\gamma}}}
\newcommand{\tdelta}{{\tilde{\delta}}}
\newcommand{\hgamma}{{\hat{\gamma}}}

\newcommand{\per}{{\mathop{\mathrm{Per}}}}

\title[NCCC for toric surfaces]{The nonequivariant coherent-constructible correspondence for toric surfaces}
\author{Tatsuki Kuwagaki}
\address{Graduate School of Mathematical Sciences, University of Tokyo, Komaba Meguroku,
Tokyo 153-8914, Japan}
\email{kuwagaki@ms.u-tokyo.ac.jp}
\begin{document}
\maketitle
\begin{abstract}
We prove the nonequivariant coherent-constructible correspondence conjectured by Fang-Liu-Treumann-Zaslow in the case of toric surfaces. Our proof is based on describing a semi-orthogonal decomposition of the constructible side under toric point blow-up and comparing it with Orlov's theorem. 
\end{abstract}

\section{Introduction}
The nonequivariant coherent-constructible correpondence ($\NC$) is a relation between the derived category of coherent sheaves on a toric variety and the derived category of constructible sheaves on a torus. $\NC$ is discovered by Bondal \cite{Bo} and formulated in terms of microlocal sheaf theory by Fang-Liu-Treumann-Zaslow \cite{FLTZ} as follows. 

Let $M$ be a free abelian group of finite rank and $N$ be its dual free abelian group. Let further $\Sigma$ be a smooth complete fan defined in $N_\R=N\otimes_{\Z}\R$ and $X_\Sigma$ be the toric variety defined by $\Sigma$. 
We write the bounded derived category of coherent sheaves on $\Xs$ by $\Db{}{\coh\Xs}$ and the bounded derived category of constructible sheaves on $M_\R/M$ by $\Db{c}{\torus}$. 
Here constructible sheaf means $\R$-constructible sheaf in the sense of \cite[\S 8.4]{KS}. We define $\Ls\subset T^*\torus$ as the coset of $\Lambda_\Sigma:=\bigcup_{\sigma\in \Sigma}\pa{\sigma^\perp+M}\times (-\sigma)\subset M_\R\times N_\R\cong T^*M_\R$. We write the full subcategory of $\Db{c}{\torus}$ spanned by objects whose microsupports  are contained in $\Ls$ by $\Db{c}{\torus, \Ls}$. It is known that there exists a fully-faithful functor 
\[\Ks:\Db{}{\coh\Xs}\hookrightarrow \Db{c}{\torus,\Ls}\]which will be defined in (\ref{functor}).
\begin{conjecture}[NCCC conjecture \cite{FLTZ2, Tr}]\label{conjecture}The functor $\Ks$ is an equivalence of triangulated categories
\begin{equation}
\Db{}{\coh \Xs}\cong \Db{c}{\torus, \Ls}.
\end{equation}
\end{conjecture}

This conjecture is proved in special cases (\cite{Tr,SS}, see also Theorem \ref{known}). The equivariant version of this conjecture is called the coherent-constructible correspondence and proved by Fang-Liu-Treumann-Zaslow \cite{FLTZ}. 

In this paper, we prove Conjecture \ref{conjecture} in dimension 2:
\begin{theorem} \label{intromain}
Conjecture \ref{conjecture} holds for any 2-dimensional smooth complete fans.
\end{theorem}
Our proof is based on Theorem \ref{sod} below.

\begin{theorem}\label{sod}
Let $\Sigma$ be an $n$-dimensional smooth complete fan and $\HSigma$ be its blow-up at a torus fixed point. Then there exists a semi-orthogonal decomposition 
\begin{equation}
\Db{c}{\torus, \Lss}\cong \hull{\Vd \pa{p_*\C_{(n-1)\cdot Z}},..., \Vd \pa{p_*\C_{Z}}, \Db{c}{\torus, \Ls}}.
\end{equation}
\end{theorem}
Here $p\colon M_\R\ra \torus$ is the quotient map, $\Vd$ is the Verdier duality functor, and $Z$ is a locally closed subset of $M_\R$ which will be defined in (\ref{Z}). 
This formula is an analogue of Orlov's theorem on the semi-orthogonal decomposition of derived category of coherent sheaves under blowing-up \cite{Or, BOr}. In the situation of Theorem \ref{sod}, Orlov's theorem says that the derived category of coherent sheaves on $\Xss$ has a semi-orthogonal decomposition
\begin{equation}\label{orlov}
\Db{}{\coh\Xss}\cong \hull{\OX_E\pa{(n-1)E}, \dots, \OX_E(E), \pi^*\Db{}{\coh \Xs}},
\end{equation}
where $\pi\colon \Xss\rightarrow \Xs$ is the blow-up morphism and $E$ is the exceptional divisor. In Lemma \ref{image}, we will prove $\Kss\pa{\OX_E(kE)}\cong \Vd\pa{p_*\C_{kZ}}[-n]$. Then one can identify Theorem \ref{sod} with Orlov's theorem via NCCC. Comparing the semi-orthogonal components, we have the following:
\begin{theorem}\label{iff}
Let $\Sigma$ be an $n$-dimensional smooth complete fan and $\HSigma$ be its blow-up at a torus fixed point.
Conjecture \ref{conjecture} holds for $\Sigma$ if and only if so is for $\hat{\Sigma}$.
\end{theorem}
In the case of toric surfaces, toric MMP and Theorem \ref{iff} allow us to reduce Conjecture \ref{conjecture} to the case of $\mathbb{P}^1\times\mathbb{P}^1$ which is already proved by Treumann \cite{Tr}. 

Nadler and Zaslow \cite{NZ, Nad} identifies the derived Fukaya category of a cotangent bundle with the bounded derived category of constructible sheaves on its base space. By using this result, Fang-Liu-Treumann-Zaslow \cite{FLTZ2} relates NCCC with homological mirror symmetry. Combining their results with our results, we obtain a version of homological mirror symmetry for toric surfaces:
\begin{corollary}\label{corhms}
Let $\Sigma$ be a $2$-dimensional smooth complete fan. Then there exists an equivalence of triangulated categories
\begin{equation}
\Db{}{\coh \Xs}\cong\D\mathrm{Fuk}\pa{T^*\torus, \Ls}.
\end{equation}
\end{corollary}
\noindent The notations will be explained in Section 2.

This paper is organized as follows. In Section 2, we briefly recall microlocal sheaf theory of Kashiwara-Schapira \cite{KS}. In Section 3, we review an aspect of the $\NC$ and collect notations. In Section 4, we prove Theorem \ref{sod}. Finally, in Section 5, we give a proof of Theorem \ref{intromain} and \ref{iff}.

\section*{Acknowledgements}
I would like to express my gratitude to Professor Kazushi Ueda and Professor Shinobu Hosono for many comments and  encouragements. I would like to thank Yuichi Ike and Fumihiko Sanda for having many stimulating discussions and giving many comments on this paper.  I also thank Ike and Takahiro Saito for answering my questions on microlocal sheaf theory many times, Daisuke Inoue and Makoto Miura for many comments at the seminars. I also would like to thank for the referee for useful comments. This work was supported by the program for Leading Graduate Schools, MEXT, Japan.

\section{Backgrounds from microlocal sheaf theory}
Let $Y$ be a differentiable manifold and $\Db{}{\Sh Y}$ be the derived category of $\C_Y$-module sheaves on $Y$. In this section, we always assume that  $n\geq 2$.

\begin{definition}[{\cite[Defintion 5.1.2]{KS}}]
For $\E\in \Db{}{\Sh Y}$, the {\em microsupport} $\ms(\E)$ of $\E$ is a closed subset of $T^*Y$ defined as follows;
for $(x,\xi)\in T^*Y$, $(x,\xi)$ is not contained in $\ms(\E)$ if there exists an open neighbourhood $V$ of $(x,\xi)$ such that for any $C^1$ function $\psi$ with $\mathrm{Graph}(d\psi)\subset V$ and  
\begin{equation}
(\R\Gamma_{\{y|\psi(y)\geq\psi(x)\}}\E)_x\simeq 0.
\end{equation}
\end{definition}
The microsupport detects the direction where the cohomology of the sheaf does not extend isomophically.

For a cone $\gamma\subset N_\bR$, we define the dual cone $\gamma^\vee$ as
\begin{equation}
\gamma^\vee:=\set{m\in M_\R \relmid \hull{m, n}\geq 0 \text{ for any }n\in \gamma}.
\end{equation}
For a subset $Z\subset Y$, we write the interior of $Z$ by $\Int(Z)$ and the relative interior of $Z$ by $\Relint(Z)$.
We say a closed convex cone $\gamma$ is {\em proper} (or strongly convex) if it satisfies $\gamma\cap (-\gamma)=\{0\}$.

\begin{proposition}[{\cite[Proposition 5.1.1]{KS}}]\label{msanother}
For $\E\in \Db{}{\Sh \R^n}$ and $(x_0,\xi_0)\in T^*\R^n$, $(x_0, \xi_0)\in \ms(\E)$ is equivalent to the following: For any neighbourhood $V$ of $x_0$, any positive integer $\epsilon$, and any proper convex cone $\delta$ with $\xi_0\in \Int(\delta^\vee)$, there exists $x\in V$ such that 
\begin{equation}
\RG{H\cap (x-\delta), \E}\xrightarrow{\not\simeq} \RG{L\cap (x-\delta), \E}
\end{equation}
where
\begin{align}
H:=\set{y\in \R^n\relmid \hull{y-x_0, \xi_0}\geq -\epsilon},\\
L:=\set{y\in \R^n\relmid \hull{y-x_0, \xi_0}=-\epsilon}.
\end{align}

\end{proposition}

\begin{theorem}[the non-characteristic deformation lemma {\cite[Proposition 2.7.2]{KS}}]\label{nonchara}
Let $I$ be an open interval in $\R$, $\{U_s\}_{s\in I}$ be a family of open subsets in Y, and $\E\in \D^\ueb(\Sh Y)$  satisfying the following:
\begin{enumerate}[label={\upshape(\arabic*)}]
\item $U_s=\displaystyle\bigcup_{t<s}U_t$ for any $s\in I$.
\item $U_t\bs U_s$ is relatively compact for any $(s,t)\in I^2$.
\item $\pa{\R\Gamma_{Y\bs U_t}\E}_x\simeq 0$ for any $s\leq t$ and $x\in \bigcap_{u> s}\mathrm{Cl}\pa{U_u\bs U_s}\bs U_t$ where $\mathrm{Cl}$ denotes taking closure.
\end{enumerate}
Then, we have $\RG{\bigcup_{t\in I}U_t,\E}\xrightarrow{\sim}\RG{U_s,\E}$ for any $s\in I$.
\end{theorem}\noindent
Theorem \ref{nonchara} holds even if $Y$ is not a manifold but $Y$ is simply a Hausdorff topological space.

For a subset $Z$ of $Y$, we define the {\em strict normal cone} $N_xZ\subset T_xY$ of $Z$ at $x\in Y$ as follows; the tangent vector $\xi\in T_xY$ is not contained in $N_xZ$ if there exists a local coordinate $U$ of $x$ and two sequences $\{x_i\}_{i\in \mathbb{N}}\subset U\bs Z$ and $\{y_i\}_{i\in \mathbb{N}}\subset U\cap Z$ satisfying convergence conditions; $x_i, y_i\ra x$ and $x_i-y_i/|x_i-y_i|\ra \xi/|\xi|$ in $U$. We define the {\em conormal cone} $N_x^*Z$ of $Z$ at $x$ as the dual cone $N_xZ^\vee$ of $N_xZ$.

\begin{lemma}\label{cone1}
Let $\gamma$ be a closed convex cone in $\R^n$. Then we have
\begin{enumerate}
\item $N_0^*\gamma=\gamma^\vee$, and
\item $N_0^*(\R^n\bs \gamma) =-\gamma^\vee$.
\end{enumerate}
\end{lemma}
\begin{proof}
These are clear from the definition of conormal cone.
\end{proof}

We say a subset $Z$ of $\R^n$ is {\em polyhedral} if it is defined by finite linear inequalities.
\begin{lemma}\label{sumconormal}
Let $Z, W \subset \R^n$ be a polyhedral subset. Then we have
\eqn{N_x^*(Z\cap W)=N_x^*Z+N_x^*W}
for $x\in Z\cap W$.
\end{lemma}
\begin{proof}
This is also clear from the definition of conormal cone.
\end{proof}

Figure \ref{exofconormal} shows some examples of Lemma \ref{cone1}.

\begin{figure}
\begin{center}
\begin{tabular}{cc}

\begin{minipage}{0.4\hsize}
\begin{tikzpicture}
\draw (1.8,0)--(0,0)--(-1.3,1.3);
\fill[gray, opacity=.2]  (1.8,0)--(0,0)--(-1.3,1.3)--(-1,-1.3)--cycle;
\draw (0,2)--(0,0)--(1.5,1.5);
\node[below] at (0,0) {$x_0$};
\node[above] at (-0.5,-0.9) {$\R^2\bs Z$};
\fill[gray, opacity=.5] (0,2)--(0,0)--(1.5,1.5)--cycle;
\fill (1.6,0) circle(1.5pt);
\draw(1.6,0)--(1.6,2);
\node[above] at (0.6, 1.0) {$N^*_{x_0}(Z)$};
\node[right] at (1.6,1) {$N^*_{x_1}(Z)$};
\node[below] at (1.6,0) {$x_1$};
\end{tikzpicture}
\end{minipage}

\begin{minipage}{0.2\hsize}
\begin{tikzpicture}
\draw (-1,0)--(0,0)--(-0.8,-0.8);
\fill[gray, opacity=.2]  (-1,0)--(0,0)--(-0.8,-0.8)--cycle;
\node[below] at (-0.9,0) {$\R^2\bs Z$};
\node[below] at (0,-0.2) {$x_2$};
\draw (1,-1)--(0,0)--(0,1.5);
\fill[gray, opacity=.5](1,-1)--(0,0)--(0,1.5)--cycle;
\node at (0.7,0) {$N_{x_2}^*(Z)$};
\end{tikzpicture}
\end{minipage}
\end{tabular}
\caption{Some examples of conormal cone\label{exofconormal}}
\end{center}
\end{figure}
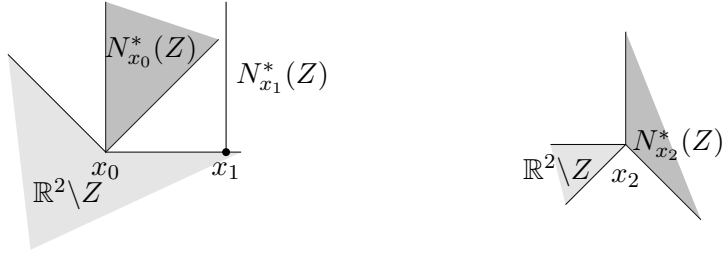

\begin{lemma}[{\cite[Corollary 5.4.9]{KS}}]\label{cone2}
We use the same notation as in the previous lemma. For $\E\in \Db{}{\Sh Y}$, we assume that $\ms(\E)\cap N_x^*Z\subset\{(x,0)\}$. Then, we have $\pa{\R\Gamma_Z(\E)}_x\simeq 0$.
\end{lemma}


\begin{lemma}\label{edgevanishing}
Let $\gamma \subset \R^n$ be an n-dimensional closed polyhedral convex cone. 
We assume that there exists a proper $n$-dimensional closed polyhedral convex cone $\delta \subset\R^n$ such that $-\delta \cap \gamma=\{0\}$. 
For $\E\in \Db{}{\Sh \R^n}$, we further assume that there exists a neighbourhood $U$ of $0$ such that $\ms(\E)\cap U\times \pa{\Relint(\gamma^\vee) +\delta^\vee} = \varnothing$ under the canonical identification $T^*\R^n\cong \R^n\times \R^n$.
Then, we have $\pa{\R\Gamma_{\gamma}(\E)}_0\simeq 0$.
\end{lemma}
\begin{proof}
We take a family of closed convex polyhedral cones $\{\gamma_s\}_{s\in [0,\infty)}$ having the following properties;
\begin{enumerate}[label=(\roman*)]
\item $\gamma\bs\{0\}\subset \mathrm{Int}(\gamma_s)$ for any $s\in [0,\infty)$,
\item $\gamma_s \subset \gamma_t$ for $s\geq t$,
\item $\bigcap_{s\in [0,\infty)}\gamma_s=\gamma$,
\item $-\delta\cap\gamma_t=\{0\}$.
\end{enumerate}
Take $x\in \Int(\delta)$, then $0\in (-\Int(\delta)+cx)\cap \gamma_t$ for any $c\in (0,1)$ and $t\in [0,\infty)$. The assumption (iv) tells us that $\{(-\delta+cx)\cap \gamma_t\}_{c\in (0,1)}$ forms a neighbourhood system of $0$ in $\gamma_t$. Hence we can take sufficiently small $c$ such that $(-\delta+cx)\cap \gamma_t\subset U$. In the following, we rewrite $cx$ as $x$.

We define 
\begin{equation}
V_s:=
\begin{cases}
\varnothing &\text{for $s<0$}, \\
(\Int(-\delta)+sx)\cap U &\text{for $s\in[0,1)$}.
\end{cases}
\end{equation}
Fix some $s\in[0,1]$ and $t\in[0,\infty)$ we define
\eqn{U_u:=V_{us}\cup (V_s\bs\gamma_t)}
for $u\in (-1,1)$. Here $V_{us} =\varnothing$ for $us<0$. Then we obtain a family of open subsets $\mathcal{U}:=\{U_u\}_{u\in (-1,1)}$. Since $(t-s)\cdot x \in \Int(\delta)$ for $s< t \in [0,1)$, we have $V_s\subset V_t$. Hence we have (1) of Theorem \ref{nonchara} for $\mathcal{U}$. 

We also have (2) of Theorem \ref{nonchara} since $(-\delta+x)\cap \gamma_t\subset U$. 

For any $a\leq b$, the set $\bigcap_{c> a} \Cl{U_c\bs U_a}\bs U_b$ is empty or $\partial V_{us}\cap \gamma_t$. At $x\in \partial V_{us}\cap \gamma_t$, we have
\begin{equation}\begin{split}
N_x^*(\R^n\bs U_u)&=N_x^*((\R^n\bs V_{us})\cap (\R^n\bs (V_s\bs\gamma_t)))\\
&=N_x^*(\R^n\bs V_{us})+N_x^*(\R^n\bs(V_s\bs \gamma_t))\\
&=-N_x^*(V_{us})+N_x^*(\gamma_t)\\
&\subset \delta^\vee + \gamma_t^\vee
\end{split}
\end{equation}
by Lemma \ref{cone1} and Lemma \ref{sumconormal}. By the assumption (i), $\gamma_t^\vee\bs\{0\}$ is contained in $\Relint(\gamma^\vee)$ for $t\in [0,\infty)$. Hence, we have $U\times \gamma_t^\vee\cap \ms(\E)\subset\{0\}$. Since $\partial V_s\cap \gamma_t\subset U$, $\ms(\E)\cap U\times (\delta^\vee+ \gamma_t^\vee)=\{0\}$. Then, by Lemma \ref{cone2}, $\{U_u\}_{u\in (-1,-1)}$ satisfy (3) of Theorem \ref{nonchara}. Hence we can use Theorem \ref{nonchara} and we have 
\eqn{\R\Gamma(V_s,\E)\xrightarrow{\simeq}\R\Gamma(V_s\bs\gamma_t,\E).}
In other words, we have $\RGL{\gamma_t}{V_s,\E}\simeq 0$ for any $s$ and $t$.

Taking the limit with respect to $t$, we have $\R\Gamma_{\gamma}(V_s,\E)\simeq 0$ for any $s\in [0,1)$.
Since $-\delta\cap\gamma=\{0\}$, we know that $\{V_s\cap \gamma\}_{s\in[0,1)}$ is an open neighbourhood system of $0$ in $\gamma$. Finally, we have $\pa{\R\Gamma_\gamma(\E)}_0\simeq 0$.
\end{proof}
Let $E\cong \R^n$ be a vector space  and $\gamma \subset E$ be an $n$-dimensional closed polyhedral convex cone. For a face $\tau$ of $\gamma^\vee$, we define a subset $R_\tau$ of $E$ by
\begin{equation}
R_\tau:=\bigcap_{\substack{\tau\prec\sigma\prec\gamma^\vee\\ \tau\neq \sigma}}\pa{\tau^\perp\bs\sigma^\perp}
\end{equation}
where $\tau\prec \sigma$ means $\tau$ is a face of $\sigma$.

\begin{lemma}\label{nonvanishing}
Let $\gamma\subset E\cong \R^n$ be an $n$-dimensional closed polyhedral convex cone. Let further $\delta$ be an $n$-dimensional proper polyhedral closed convex cone. For $\E\in \Db{}{\Sh E}$, we assume the following:
\begin{enumerate}[label={\upshape(\roman*)}]
\item There exists a neighbourhood $V$ of $0$ such that $\pa{V\bs\gamma}\times (\delta^\vee+\gamma^\vee)\cap \ms(\E)\subset \pa{V\bs\gamma}\times \set{0}\cup \bigsqcup_{0\neq\tau\prec\gamma^\vee}R_\tau\times\tau$. \label{281}
\item There exists $\xi_0\in \Relint(\gamma^\vee)\cap \Int(\delta^\vee)$ such that $(0,\xi_0)\in \ms(\E)$.\label{282}
In particular, $\gamma\cap(-\delta)=\set{0}$.\label{283}
\item $\pa{\gamma^\vee\bs\Relint(\gamma^\vee)}\cap \delta^\vee=\set{0}$.\label{284}
\item $\Int(\gamma)\cap \Int(\delta)\neq \varnothing$.\label{285}
\end{enumerate}
Then, for any neighbourhood $U$ of $0$, there exists $x\in U$ such that $\RGL{\gamma}{-\delta+x,\E}\not\simeq0$.
\end{lemma}
\begin{proof}

We note that \ref{283} and \ref{285} imply a family of closed subsets $\set{(x-\delta)\cap\gamma\relmid x\in \Int(\gamma)\cap \Int(\delta)}$ forms a neighbourhood system of $0$ in $\gamma$. Combining with $\xi_0\in \Int(\delta^\vee)$ from \ref{282}, we can take a sufficiently small $\epsilon$ and $x_0\in \Int(\gamma)\cap \Int(\delta)$ such that $\overline{W}:=\pa{x_0-\delta}\cap H\subset V\cap U$ where 
\begin{equation}
H:=\set{y\in V\relmid \hull{y, \xi_0}\geq -\epsilon}.
\end{equation}

We define an open neighbourhood $W$ of $0$ by $W:=\Int\pa{\overline{W}}$.
Proposition \ref{msanother} and the property $(0,\xi_0)\in \ms(\E)$ from \ref{282} ensures that there exists $x\in W$ such that 
\begin{equation}\label{defisom}
\RG{H\cap (x-\delta), \E}\xrightarrow{\not\simeq} \RG{L\cap (x-\delta), \E}
\end{equation}
where
\begin{align}
L:=\set{y\in V\relmid \hull{y, \xi_0}=-\epsilon}.
\end{align}
By the definition of $W$, it follows that $H\cap(x-\delta), L\cap(x-\delta)\subset V$.

Let us assume that $\pa{x-\delta}\cap \gamma$ were empty. Then take $y_0\in \Relint((x-\delta\cap\gamma)\cap L)$. From $\xi_0\in \Relint(\gamma^\vee)$ of \ref{282}, we can define a family of open subsets $\mathcal{V}:=\set{V_s}_{s\in(-1, 1+\alpha)}$ of $H$ with 
\begin{equation}
V_{s}:=
\pa{sx+\pa{1-s}y_0-\Int(\delta)}\cap H \text{ when $s\in (0,1+\alpha)$}
\end{equation}
where $\alpha>0$ is taken to be sufficiently small to satisfy $V_{1+\alpha}\subset V\bs \gamma$. The family $\mathcal{V}$ clearly satisfies (1) and (2) of Theorem \ref{nonchara}. For $s\geq t$, the subset $\bigcap_{u>s}\Cl{V_u\bs V_s}\bs V_t$ is $\varnothing$ or $\pa{sx+(1-s)y_0-\partial\delta}\cap H$.
Hence the normal cone $N_{y}^*(V_{1+\alpha}\bs V_s)$ is contained in $\xi_0+\delta^\vee=\delta^\vee$ for $y\in \bigcap_{u>s}\Cl{V_u\bs V_s}\bs V_t$. Since $V_{1+\alpha}\subset V\bs \gamma$, \ref{281} and \ref{284} tell us that $\ms(\E)\cap N_{y}^*(V_{1+\alpha}\bs V_s)=\set{y}\times\set{0}$ for $y\in \bigcap_{u>s}\Cl{V_u\bs V_s}\bs V_t$. By Lemma \ref{cone2}, $\mathcal{V}$ satisfies (3) of Theorem \ref{nonchara}. Thoerem \ref{nonchara} says that 
\begin{equation}
\RG{V_{1+\alpha},\E}\xrightarrow{\simeq}\RG{L\cap((1+\alpha)x-\alpha y_0-\delta),\E}
\end{equation}
for any sufficiently small $\alpha$. Taking $\alpha\rightarrow 0$, we have 
\begin{equation}
\RG{H\cap (x-\delta),\E} \xrightarrow{\simeq} \RG{L\cap (x-\delta),\E}.
\end{equation}
This contradicts to (\ref{defisom}), hence we have $\pa{x-\delta}\cap \gamma\neq\varnothing$.

We again define a family of open subsets $\mathcal{V}'_\alpha:=\set{V'_{s,\alpha}}_{s\in(-1,1)}$ of $H$ with
\begin{equation}
V'_{s,\alpha}:=
\pa{\pa{s-1}(1+\alpha)x+\pa{1-s}y_0+\pa{\pa{(1+\alpha)x-\Int(\delta)}\bs\gamma}}\cap H \text{ when $s\in (0,1)$}
\end{equation}
where $\alpha$ is taken to be sufficiently small to satisfy $V'_{1,\alpha}\subset V$.
It is clear that the family $\mathcal{V}'_\alpha$ satisfies (1) and (2) of Theorem \ref{nonchara}. For $s$, we define two sets:
\begin{align}
Z_1:&=\pa{\pa{s-1}(1+\alpha)x+\pa{1-s}y_0+\pa{\pa{(1+\alpha)x-\partial\delta}\bs\gamma}}\cap H. \\
Z_2:&=\pa{s-1}(1+\alpha)x+\pa{1-s}y_0+\pa{\pa{(1+\alpha)x-\delta}\cap\partial\gamma}.
\end{align}
Then, for $s\leq t$, $\bigcap_{u>s}\Cl{V'_{u,\alpha}\bs V'_{s,\alpha}}\bs V'_{t,\alpha}$ is $\varnothing$ or $Z_1\sqcup Z_2$. If $y\in Z_1$, the normal cone $N^*_y(V_{1,\alpha}'\bs V'_{s,\alpha})$ is contained in $\delta^\vee$, hence we have $\ms(\E)\cap N^*_y(V_{1,\alpha}'\bs V'_{s,\alpha})=\set{y}\times\set{0}$ by \ref{281} and \ref{284}.
On the other hand, we have a decomposition $Z_2=\bigsqcup_{\tau\precneqq\gamma^\vee}W_\tau$ where $W_\tau$ is defined by
\begin{align}
W_\tau:=\pa{s-1}(1+\alpha)x+\pa{1-s}y_0+\pa{\pa{(1+\alpha)x-\delta}\cap R_\tau\cap \gamma}.
\end{align} 
Since $y_0-x\in -\Int(\gamma)$, we have $\hull{y_0-x, \gamma^\vee\bs\set{0}}<0$. Then we have $W_\tau\cap \bigcup_{\tau\prec\sigma\prec\gamma^\vee}R_\sigma\cup\bigcup_{\sigma\prec\tau}R_\sigma=\varnothing$. Moreover, if $\sigma$ is not a face of $\tau$ and $\tau$ is not a face of $\sigma$, $\sigma\cap\tau$ is a proper face of $\sigma$ and $\tau$. For $y\in W_\tau$, we have
\begin{equation}
N_{y}^*(V_{1,\alpha}'\bs V'_{s,\alpha})= \set{y}\times \pa{f\pa{\delta^\vee}+\tau}
\end{equation}
where $f\pa{\delta^\vee}$ is one of the faces of $\delta^\vee$. By \ref{284}, we can see $\sigma\cap (\delta^\vee+\tau)$ is a proper face of $\sigma$.
Taking intersection of $N_y^*\pa{V_{1,\alpha}'\bs V'_{s,\alpha}}$ with the inclusion relation of \ref{281}, we have 
\begin{equation}
\begin{split}
\ms(\E)\cap \Relint(N_y^*\pa{V'_{1,\alpha}\bs V'_{s,\alpha}})=\varnothing.
\end{split}
\end{equation}
By Lemma \ref{edgevanishing}, the family $\mathcal{V}'_\alpha$ satisfies (3) of Theorem \ref{nonchara} and we have
\begin{equation}\label{alpha}
\RG{H\cap V_{1,\alpha}',\E}\xrightarrow{\simeq}\RG{L\cap ((1+\alpha)x-\alpha y_0-\delta),\E}.
\end{equation}
Take $\alpha\rightarrow 0$, then (\ref{alpha}) becomes
\begin{equation}
\RG{H\cap (x-\delta)\bs\gamma, \E}\xrightarrow{\simeq}\RG{L\cap (x-\delta),\E}.
\end{equation}
Combining this with (\ref{defisom}), we have
\begin{equation}
\RG{H\cap(x-\delta),\E}\xrightarrow{\not\simeq}\RG{H\cap (x-\delta)\bs\gamma,\E}.
\end{equation}
Hence we have $\RGL{\gamma}{H\cap (x-\delta),\E}\not\simeq 0$. Since $\gamma\subset H$, we conclude $\RGL{\gamma}{x-\delta,\E}\not\simeq 0$.
\end{proof}

\section{A review of NCCC}
Let $M$ be a free abelian group of rank $n$ and $N$ be the dual free abelian group of $M$. We consider a smooth complete fan $\Sigma$ defined in $
N_{\R}:=N\otimes_\mathbb{Z}\R$ and the associated toric variety $X_{\Sigma}$. For $\sigma\in \Sigma$, we write the corresponding affine toric subvariety of $\Xs$ by $U_{\sigma}$ and the open immersion by $i_{\sigma}:U_{\sigma}\hookrightarrow X_{\Sigma}$. The {\it theta quasi-coherent sheaf} associated to $\sigma\in \Sigma$ is 
\begin{equation}
\qtheta(\sigma):=\OX_{\sigma}:=i_{\sigma*}\mathcal{O}_{U_\sigma} \in \D^{\mathrm{b}}(\Qcoh \Xs)
\end{equation} 
where $\D^\ueb(\Qcoh \Xs)$ is the bounded derived category of quasi-coherent sheaves.
It is known that $\D^\ueb(\coh\Xs)\subset \hull{\qtheta(\sigma)}_{\sigma\in\Sigma}$, where $\hull{\cdot}$ denotes the generated full subcategory \cite[Proposition 2.6]{Tr}.

We define the {\it theta quasi-constructible sheaf} associated to $\sigma\in\Sigma$ as
\begin{equation}
\ctheta(\sigma):=p_{!}\Vd(\C_{\sigma^\vee})\in \D^{\ueb}_{qc}(M_\R/M)
\end{equation}
where $p\colon M_\R\rightarrow M_\R/M$ is the quotient map, $\Vd\colon \D^\ueb_{qc}(M_\R)\rightarrow \D^{\ueb}_{qc}(M_\R)^{op}$ is the
Verdier duality functor, $\mathbb{C}_{\sigma^{\vee}}$ is the zero-extension of the constant sheaf on $\sigma^{\vee}$, and $\D^\ueb_{qc}(\torus)$ is the bounded derived category of quasi-constructible sheaves of $\C$-modules. Here, quasi-constructible (weakly constructible in \cite{KS}) means that it is locally constant along some stratification but not necessarily of finite rank.

For $m\in\tau^\vee$, one can define $\theta'_m \in \Hom^0_{\Db{}{\Qcoh\Xs}}\pa{\qtheta(\sigma),\qtheta(\tau)}$ the multiplication by the character $\chi^m$;
\begin{equation}
\theta_m'\colon\qtheta(\sigma)\xrightarrow{\times\chi^m} \qtheta(\tau).
\end{equation}
This correspondence induces an isomorphism 
\begin{equation}
\Hom^i_{\D^\ueb(\Qcoh\Xs)}\pa{\qtheta(\sigma),\qtheta(\tau)}\cong 
\begin{cases}
\mathbb{C}[\tau^\vee\cap M]&\text{when $\sigma\supset\tau$ and $i=0$}, \\
0&\text{otherwise}.
\end{cases}
\end{equation}

Similarly, for $m\in\tau^\vee$, one can define $\theta_m\in \Hom^0_{\Db{c}{\torus}}\pa{\ctheta(\sigma),\ctheta(\tau)}$ as the composition 
\begin{equation}
\theta_m\colon\ctheta(\sigma)=p_!\Vd(\C_{\sigma^\vee})=p_!\Vd(\C_{\sigma^\vee+m})\xrightarrow{p_!\Vd\left(r^{\tau^\vee}_{\sigma^\vee+m}\right)}\ctheta(\tau)
\end{equation}
where $\chi^m$ is the character corresponding to $m$ and $r^{\tau^\vee}_{\sigma^\vee+m}:\C_{\tau^\vee}\ra\C_{\sigma^\vee+m}$ is the restriction map. This correspondence induces an isomorphism
\begin{equation}
\Hom^i_{\D^\ueb_c(\torus)}\pa{\ctheta(\sigma),\ctheta(\tau)}\cong\begin{cases}
\mathbb{C}[\tau^\vee\cap M]&\text{when $\sigma\supset\tau$ and $i=0$}, \\
0&\text{otherwise}.
\end{cases}
\end{equation}

The category $\Gamma\pa{\Ls}$ is a dg-category whose set of objects is $\Sigma$ and Hom-spaces are defined by
\begin{equation}
\mathrm{hom}^i_{\Gamma\pa{\Ls}}(\sigma,\tau):=\begin{cases}
\C[\tau^\vee\cap M] & \text{when $\sigma\supset \tau$ and $i=0$,}\\
0&\text{otherwise.}
\end{cases}
\end{equation}
with trivial differentials.　
We write the full sub dg-category of $\D^{dg}\pa{\Qcoh\Xs}$ (resp. $\D^{dg}_{qc}\pa{\torus,\Ls}$) 
spanned by $\set{\qtheta(\sigma)}_{\sigma\in\Sigma}$ (resp. $\set{\ctheta(\sigma)}_{\sigma\in \Sigma}$) by $\qtheta_{dg}$ (resp. $\ctheta_{dg}$). 
Then we have two quasi-equivalences of dg-categories
\begin{align}
\Gamma\pa{\Ls}\rightarrow \qtheta_{dg}, \\
\Gamma\pa{\Ls}\rightarrow\ctheta_{dg}.
\end{align}
Hence, we also have the quasi-equivalence of perfect dg-modules of dg-categories $\per_{dg}\qtheta_{dg}\simeq \per_{dg}\ctheta_{dg}$.　We write the equivalence induced on the homotpy catogories by
\begin{equation}
K_{\Sigma}\colon \hull{\qtheta}_{\sigma\in\Sigma}\xrightarrow{\cong}H^0\pa{\per_{dg}\qtheta_{dg}}\xrightarrow{\cong}H^0\pa{\per_{dg}\ctheta_{dg}}\xrightarrow{\cong}\hull{\ctheta}_{\sigma\in\Sigma}
\end{equation}
By the definition, $K_\Sigma$ sends $\qtheta(\sigma)$ to $\ctheta (\sigma)$ and $\theta'_m$ to $\theta_m$ \cite[Theorem 2.3]{Tr}.

To describe $\NC$, we identify $M_{\R}\times N_{\R}\cong T^*M_{\R}$ and define 
\begin{equation}
\Lambda_{\Sigma}:=\bigcup_{\sigma\in \Sigma}\pa{\sigma^{\perp}+M}\times (-\sigma)\subset T^{*}M_{\R}
\end{equation}
and $\overline{\Lambda_{\Sigma}}\subset M_\R/M\times N_\R \cong T^*\torus$ as the image of the ${\Lambda_\Sigma}$ under the projection $\tilde{p}:T^*M_\R\ra T^*\torus$. 
We write the full subcategory of $\D^\ueb_c(\torus)$ whose objects have microsupports in $\Ls$ by $\Db{c}{\Xs, \Ls}$. Treumann showed that the essential image of $\Db{}{\coh\Xs}$ by $K_\Sigma$ is contained in $\Db{c}{\torus, \Ls}$ \cite[Proposition 2.7]{Tr}. Hence we obtain the following functor
\begin{equation}\label{functor}
\Ks:={K_\Sigma}|_{\Db{}{\coh\Xs}}\colon\Db{}{\coh\Xs}\rightarrow\Db{c}{\torus,\Ls}.
\end{equation}
The functor $\Ks$ is fully-faithful, since it is a restriction of an equivalence.

Conjecture \ref{conjecture} is motivated by homological mirror symmetry. By the result of Nadler-Zaslow \cite{NZ} and Nadler \cite{Nad}, the derived Fukaya category $\D\mathrm{Fuk}(\torus)$ of $\torus$  and the bounded derived category of constructible sheaves $\Db{c}{\torus}$ on $\torus$ are equivalent;

\begin{theorem}[Nadler-Zaslow \cite{NZ}, Nadler \cite{Nad}]\label{hms}

For a real analytic manifold $X$, there exists an equivalence of triangulated categories
\begin{equation}
\Db{c}{X} \cong \D\mathrm{Fuk}(T^*X).
\end{equation}
\end{theorem}
\noindent Fang-Liu-Treumann-Zaslow \cite{FLTZ2} defined $D^b\mathrm{Fuk}(T^*\torus, \Ls)$ as the essential image of $D^b_c(\torus, \Ls)$ under Nadler-Zaslow's equivalence. Combining Theorem \ref{hms} with Conjecture \ref{conjecture}, we have a version of homological mirror symmetry for toric varieties.

There are some results on Conjecture $\ref{conjecture}$. We say a smooth complete fan $\Sigma$ is a zonotopal unimodular fan if $\Sigma$ is obtained from a hyperplane arrangement and any linearly independent subset of the set of ray generators of $\Sigma$ can be extended to $\Z$-basis of $N$. We say a smooth complete fan $\Sigma$ is cragged when the following two conditions are satisfied:
\begin{enumerate}
\item For any subset $S$ of $\Sigma$, the cone hull of $S$ is a union of a subset of $\Sigma$.
\item For any linearly independent subset $B$ of the set of ray generators $R$, the lattice generated by $\mathrm{Cone}(B)\cap R$ has $B$ as a $\Z$-basis.
\end{enumerate}

\begin{theorem}[Scherotzke-Sibilla \cite{SS}, Treumann \cite{Tr}]\label{known}
Let $\Sigma$ is a smooth complete fan.
\begin{enumerate}
\item Conjecture \ref{conjecture} holds when $\Sigma$ is zonotopal unimodular \cite[Corollary 4.5]{Tr}. 
\item Conjecture \ref{conjecture} holds when $\Sigma$ is cragged \cite[Theorem 6.11]{SS}.
\end{enumerate}
\end{theorem}

Both $\Db{}{\coh\Xs)}$ and $\Db{c}{\torus,\Ls}$ carry monoidal stuctures as follows;
\begin{enumerate}
\item $\otimes:=\otimes^{\mathbb{L}}$; derived tensor product in $\Db{}{\coh\Xs)}$,
\item $\star:=m_!\circ\boxtimes^{\mathbb{L}}$; the composition of 
\begin{align*}
\Db{c}{\torus,\Ls}\times\Db{c}{\torus,\Ls}&\\
\xrightarrow{\boxtimes^{\dL}}&\Db{c}{\torus\times \torus,\Ls\times\Ls}\\
\xrightarrow{m_!}& \Db{c}{\torus,\Ls}
\end{align*}
where $\boxtimes^{\dL}$ is the exterior tensor product and $m$ is the multiplication map with respect to the group structure of $\torus$.
\end{enumerate}

Some general properties of
$\kappa_{\Sigma}$ are known.
\begin{theorem}[Fang-Liu-Treumann-Zaslow \cite{FLTZ}, Treumann \cite{Tr}]\label{kappa}
Let $\Sigma$ be a smooth complete fan and $\HSigma$ be a smooth complete subdivision of $\Sigma$. Let further $\pi^{*}: \Db{}{\coh\Xs}\rightarrow\Db{}{\coh\Xss}$ be the pull-back along $\pi:\Xss\ra \Xs$ which is the morphism associated to the subdivision. Then we have the following:
\begin{enumerate}[label={\upshape(\arabic*)}]
\item There exists a natural equivalence $\Kss\circ \pi^*\cong \iota\circ \Ks$ where $\iota$ is the inclusion functor $\Db{c}{\torus, \Ls}\rightarrow \Db{c}{\torus, \Lss}$ implied by $\Ls\subset\Lss$.
\item The functor $\Ks$ is monoidal with respect to the monoidal structures $\otimes$ and $\star$.
\item There exists a natural equivalence, $\Vd\circ \Ks\cong\alpha^*\circ\Ks\circ D$ where $D:=\R\mathcal{H}om(-, \OX_{\Xs})$ and $\alpha: \torus \ra \torus$ is the inversion map $x\mapsto -x$.
\end{enumerate}
\end{theorem}

\section{NCCC and Blow-up formula}
Let $M$ be a free abelian group with $n:=\rank M\geq 2$, $N$ be the dual of $M$, and $\Sigma$ be a smooth complete fan defined in $N_\R$. 
We write a toric blow-up of $\Xs$ centered at a torus-fixed point by $\pi:\Xss\rightarrow \Xs$, the exceptional divisor by $j:E\rightarrow\Xs$, and the
corresponding ray by $\rho_E\in \hat{\Sigma}$. We write the unique cone which corresponds to the affine toric subvariety containing the blow-up point by $\sigma_c \in \Sigma$. 
The cone $\sigma_c$ is $n$-dimensional and we write edges of $\sigma_c$ by $\rho_1,..., \rho_n$ and the ray generator of $\rho_i$ by $e_i$. Then, we have $e_E:=\sum_{i=1}^ne_i$ for the ray generator of $\rho_E$. Since $\sigma$ is smooth, $\{e_i\}_{i=1}^n$ forms a basis of $N$. The dual basis of $M$ is denoted by $\{e_i^\vee\}_{i=1}^n$.

We define a locally closed subset $Z\subset M_\R$ by 
\begin{equation}\label{Z}
Z:=\set{m\in M_\R\relmid \hull{m, e_E} \geq-1}\cap\set{m\in M_\R\relmid \hull{m, e_i}<0\hspace{2mm}\text{for any $i$}}.
\end{equation}

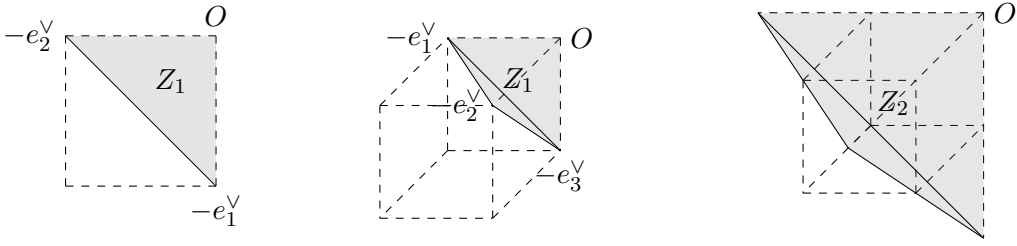
\begin{figure} 
\begin{center}
\begin{tabular}{ccc}
\begin{minipage}{0.3\hsize}
\begin{tikzpicture}
\draw (2,0) -- (0,2);
\draw[dashed] (0,2) -- (2,2);
\draw[dashed] (2,0) -- (2,2);
\draw[dashed] (0,0) -- (0,2);
\draw[dashed] (0,0) -- (2,0);
\fill[gray, opacity=.2] (0,2) -- (2,0) -- (2,2) --cycle;
\node at (1.4,1.4) {$Z_1$};
\node[above] at (2,2) {$O$};
\node[below] at (2,0) {$-e_1^\vee$};
\node[left] at (0,2) {$-e_2^\vee$};
\end{tikzpicture}
\end{minipage}

\begin{minipage}{0.3\hsize}
\begin{tikzpicture}[x=0.5cm,y=0.5cm,z=0.3cm, domain=-6:0]
\draw[dashed] (0,0,0)--(-3,0,0);
\draw[dashed] (0,0,0)--(0,-3,0);
\draw[dashed] (0,0,0)--(0,0,-3);
\draw[dashed] (-3,-3,0)--(-3,0,0);
\draw[dashed] (-3,-3,0)--(0,-3,0);
\draw[dashed] (0,-3,-3)--(0,0,-3);
\draw[dashed] (-3,0,-3)--(-3,0,0);
\draw[dashed] (0,-3,-3)--(0,-3,0);
\draw[dashed] (-3,0,-3)--(0,0,-3);
\draw[dashed] (-3,0,-3)--(-3,-3,-3);
\draw[dashed] (-3,-3,0)--(-3,-3,-3);
\draw[dashed] (-3,-3,-3)--(0,-3,-3);
\fill[gray, opacity=.1](0,0,0)--(-3,0,0)--(0,-3,0)--cycle;
\fill[gray, opacity=.1](0,0,0)--(0,-3,0)--(0,0,-3)--cycle;
\fill[gray, opacity=.1](0,0,0)--(-3,0,0)--(0,0,-3)--cycle;
\fill[gray, opacity=.1](-3,0,0)--(0,-3,0)--(0,0,-3)--cycle;
\node[right] at (0,0,0) {$O$};
\node[left] at (-3,0,0) {$-e_1^\vee$};
\node[below] at (0,-3,0) {$-e_3^\vee$};
\node[left] at (0,0,-3) {$-e_2^\vee$};
\node at (-0.7,-0.7,-0.7) {$Z_1$};
\draw (-3,0,0)--(0,-3,0);
\draw (-3,0,0)--(0,0,-3);
\draw (0,0,-3)--(0,-3,0);
\end{tikzpicture}
\end{minipage}
\begin{minipage}{0.3\hsize}
\begin{tikzpicture}[x=0.5cm,y=0.5cm,z=0.3cm, domain=-6:0]
\draw[dashed] (0,0,0)--(-6,0,0);
\draw[dashed] (0,0,0)--(0,-6,0);
\draw[dashed] (0,0,0)--(0,0,-6);
\draw[dashed] (-3,-3,0)--(-3,0,0);
\draw[dashed] (-3,-3,0)--(0,-3,0);
\draw[dashed] (0,-3,-3)--(0,0,-3);
\draw[dashed] (-3,0,-3)--(-3,0,0);
\draw[dashed] (0,-3,-3)--(0,-3,0);
\draw[dashed] (-3,0,-3)--(0,0,-3);
\draw[dashed] (-3,0,-3)--(-3,-3,-3);
\draw[dashed] (-3,-3,0)--(-3,-3,-3);
\draw[dashed] (-3,-3,-3)--(0,-3,-3);
\fill[gray, opacity=.1](0,0,0)--(-6,0,0)--(0,-6,0)--cycle;
\fill[gray, opacity=.1](0,0,0)--(0,-6,0)--(0,0,-6)--cycle;
\fill[gray, opacity=.1](0,0,0)--(-6,0,0)--(0,0,-6)--cycle;
\fill[gray, opacity=.1](-6,0,0)--(0,-6,0)--(0,0,-6)--cycle;
\node[right] at (0,0,0) {$O$};

\node at (-1.5,-1.5,-1.5) {$Z_2$};
\draw (-6,0,0)--(0,-6,0);
\draw (-6,0,0)--(0,0,-6);
\draw (0,0,-6)--(0,-6,0);
\end{tikzpicture}
\end{minipage}
\end{tabular}
\caption{$Z_1$ in $n=2$, $Z_1$ and $Z_2$ in $n=3$\label{Z3}}
\end{center}
\end{figure}
\noindent
Also recall that the functor
\begin{equation}
\Kss: \Db{}{\coh \Xs}\rightarrow \Db{c}{\torus, \Lss}.
\end{equation}
To prove Theorem \ref{sod}, we prepare the following lemma.

\begin{lemma}\label{image} For $k\geq 1$,
\begin{equation}
\Kss\pa{\OX_E(kE)} \simeq \Vd(p_*\C_{Z_k})[-n]
\end{equation}
where $Z_k:=k\cdot Z$ and $p\colon M_\R\rightarrow \torus$ is the quotient map.
\end{lemma}
Some examples of the subsets $Z_k$ are depicted in Figure \ref{Z3}.

\begin{proof}
To calculate $\kappa_{\hat{\Sigma}}(\mathcal{O}_E(kE))$, we first consider the following resolution;
\begin{equation}\label{exact}
0\rightarrow \mathcal{O}\rightarrow \mathcal{O}(E)\rightarrow \mathcal{O}_E(E)\rightarrow 0
\end{equation}
where we write $\mathcal{O}_{\Xss}$ by $\OX$.

Next, we take $\mathrm{\check{C}}$ech resolutions：
\begin{equation}
\xymatrix{ 
0\ar[r] &\OX \ar[r]^\epsilon\ar[d] &\OX(E) \ar[r]\ar[d]^{r_n} &\OX_E(E)  \ar[r]\ar[d]^{[r_n]}&0\\
0\ar[r] &\displaystyle\bigoplus_{\sigma\in \HSigma(n)}\OX_{U_\sigma}\ar[r]^{\epsilon_n}\ar[d]&\displaystyle\bigoplus_{\sigma\in \HSigma(n)}\OX_{U_\sigma}(E)\ar[r]\ar[d]^{r_{n-1}}&\coker(\epsilon_n)\ar[r]\ar[d]^{[r_{n-1}]} &0\\
0\ar[r] &\displaystyle\bigoplus_{\sigma\in \HSigma(n-1)}\OX_{U_\sigma}\ar[r]^{\epsilon_{n-1}}\ar@{->}[d] &\displaystyle\bigoplus_{\sigma\in \HSigma(n-1)}\OX_{U_\sigma}(E)\ar[r]\ar[d] &\coker(\epsilon_{n-1})\ar[r]\ar[d] &0\\
&\vdots&\vdots&\vdots}
\end{equation}
\noindent Hereafter, we will ignore cones in $\HSigma$ which do not contain $\rho_E$, since they are irrelevant for the calculation of $\OX_E(E)$ in the above diagram.

By definition, $\OX_{U_\sigma}=\qtheta(\sigma)$. On each affine toric variety $U_\sigma$, we can take some $e_j$ such that $\chi^{e_j^\vee}\cdot\OX_{U_\sigma}(E)=\OX_{U_\sigma}$ as a subsheaf of the constant sheaf of the  function field. We fix such
$j$ and write $e_{\sigma}^\vee:= e_j^\vee$. 
The exact triangle,
\begin{equation}
\xymatrix{
\qtheta(\sigma)\ar[r]^{\chi^{e_\sigma^\vee}}&\qtheta(\sigma)\ar[r]&\coker(\chi^{e_\sigma^\vee})\ar[r]^{\hspace{2mm}[1]}&\qtheta(\sigma)[1]
}
\end{equation}
in $\D^\ueb_{qc}(\torus)$ is sent to
\begin{equation}\label{cone}
\xymatrix{
\ctheta(\sigma)\ar[r]^{\Kss\pa{\chi^{e_\sigma^\vee}}}&\ctheta(\sigma)\ar[r]&\Kss\pa{\coker\pa{\chi^{e_\sigma^\vee}}}\ar[r]^{\hspace{6mm}[1]}&\ctheta(\sigma)[1]}
\end{equation}
by $\Kss$. 

On the other hand, there exists the following exact triangle in $\Db{qc}{\torus}$
\begin{equation}
\xymatrix{
\C_{\sigma^\vee-e_\sigma^\vee}\ar[r]&\C_{\sigma^\vee}
\ar[r]^{\hspace{-3mm}[1]}&\C_{Z_{\sigma}}[1]\ar[r]&\C_{\sigma^\vee}[1]
}
\end{equation}
where the left arrow is the restriction map and $Z_\sigma:=\pa{\sigma^\vee-e_\sigma^\vee}\bs\pa{\sigma^\vee}$. Note that $Z_\sigma$ does not depend on the choice of $e_\sigma^\vee$. Applying $p_!\circ \Vd$ to this triangle, we have 
\begin{equation}\label{above}
\xymatrix{
\ctheta(\sigma)\ar[r]^{\Kss\pa{\chi^{e_\sigma^\vee}}}&\ctheta(\sigma)\ar[r]&p_!\Vd\pa{\C_{Z_\sigma}}\ar[r]^{\hspace{6mm}[1]}&\ctheta(\sigma)[1]}.
\end{equation}
\noindent Comparing (\ref{above}) with (\ref{cone}), we have $\Kss\pa{\coker\pa{\chi^{e_\sigma^\vee}}}\cong p_!\Vd\pa{\C_{Z_\sigma}}$.

We next calculate $\Kss\pa{[r_i]}$. For $\sigma \in \Sigma(i+1)$ and $\tau\in \Sigma(i)$, the restriction map $\OX_\sigma(E)\ra\OX_\tau(E)$ is translated into $\qtheta(\sigma)\xrightarrow{\chi^{e_{\tau}^{\vee}-e_{\sigma}^{\vee}}}\qtheta(\tau)$. 
We define the sheaf $\hat{\C}_{\sigma^\vee}:=\C_{\sigma^\vee-e_\sigma^\vee}$ on $M_\R$, 
then the restriction map $\hat{\C}_{\sigma^\vee}\ra\hat{\C}_{\tau^\vee}$ is mapped to $\Kss\pa{\chi^{e_\tau^\vee-e_\sigma^\vee}}$ by applying $p_!\Vd$. 
On the other hand, we can see from the definition of $\Ks$ that the morphism induced on $\coker(\epsilon_i)\ra\coker(\epsilon_{i-1})$ is mapped to the composition of the restriction map $\C_{Z_\sigma}\ra \C_{Z_\tau}$ and $p_!\Vd$. 
By summing up the restriction maps $\hat{\C}_{\sigma^\vee}\ra\hat{\C}_{\tau^\vee}$ with $\Cech$ signs and applying $p_!\Vd$, we have $\Kss(r_i)$. Hence, we conclude that $\Kss([r_i])$ is obtained by summing up the restriction maps ${\C}_{Z_\sigma}\ra{\C}_{Z_\tau}$ with $\Cech$ signs and applying $p_!\Vd$.

Note that $\hat{\C}_\sigma$ does not depend on a specific choice of $e_\sigma$.
We can observe that $\{Z_\sigma\}_{\rho_E\subset\sigma\in\HSigma(2)}$ forms a covering of  $Z_{\rho_E}\bs Z_1$. Moreover, $\{Z_\sigma\}_{\rho_E\subset \sigma\in\HSigma}$ coincides with the $\mathrm{\check{C}ech}$ covering obtained from $\{Z_\sigma\}_{\rho_E\subset\sigma\in\HSigma(2)}$. Hence, we have 
\begin{equation}
\C_{Z_1}\simeq \pa{\bigoplus_{\sigma\in \HSigma(1)}\C_{Z_\sigma}\ra\cdots\ra\bigoplus_{\sigma\in \HSigma(n-1)}\C_{Z_\sigma}\ra\bigoplus_{\sigma\in \HSigma(n)}\C_{Z_\sigma}}
\end{equation}
where the differentials in the RHS are the $\Cech$ differentials, and we regard the first term in the RHS is in degree 0.

To sum up, we have
\begin{equation}
\begin{split}
\Vd&\pa{\C_{p(Z_1)}}[-n]\\
&\simeq p_!\Vd\pa{\C_{Z_1}}[-n]\\
&\simeq p_!\Vd\pa{\bigoplus_{\sigma\in \HSigma(1)}\C_{Z_\sigma}\ra\cdots\ra\bigoplus_{\sigma\in \HSigma(n-1)}\C_{Z_\sigma}\ra\bigoplus_{\sigma\in \HSigma(n)}\C_{Z_\sigma}}[-n]\\
&\simeq \Kss\pa{\bigoplus_{\sigma\in \HSigma(n)}\coker\pa{\chi^{e_\sigma^\vee}}[n]\ra\bigoplus_{\sigma\in \HSigma(n-1)}\coker\pa{\chi^{e_\sigma^\vee}}[n]\ra\cdots\ra\bigoplus_{\sigma\in \HSigma(1)}\coker\pa{\chi^{e_\sigma^\vee}}[n]}[-n]\\
&\simeq \Kss\pa{\OX_E(E)}.
\end{split}
\end{equation}

By Theorem \ref{kappa} (2) and (3), we have
\begin{equation}
\begin{split}
\Kss\pa{\OX_{E}(kE)}&\simeq \Vd\pa{\C_{p(Z_1)}}\star\cdots\star\Vd\pa{\C_{p(Z_1)}}[-n]\\
&\simeq\Vd\pa{p_*\C_{Z_k}}[-n].
\end{split}
\end{equation}
This completes the proof.
\end{proof}

\begin{proof}[Proof of Theorem \ref{sod}]
We note that $\Kss\pa{\OX_E(kE)}$ is exceptional for $1\leq k\leq n-1$, since $\Kss$ is fully-faithful and $\OX_E(kE)$ is exceptional. We write the triangulated hull of $\OX_E(kE)$ by $\D_k$. Then, $\Kss\pa{\D_k}$ is an admissible full subcategory of $\Db{c}{\torus, \Lss}$. 
We have a semi-orthogonal decomposition \cite{BK},
\begin{equation}
\begin{split}
&\Db{c}{\torus, \Lss} \\
&\cong \hull{\Kss\pa{\D_{-n+1}}, ..., \Kss\pa{\D_{-1}},
^{\lperp}\hull{\Kss\pa{\D_{-n+1}}, ..., \Kss\pa{\D_{-1}}}}.
\end{split}
\end{equation} 
Hence, it is enough to show that $^{\perp}\hull{\Kss\pa{\D_{-n+1}},..., \Kss\pa{\D_{-1}}}=
\Db{c}{\torus, \Ls}$.

\subsection*{Step 1 ($\subset$)}
For $\E\in \Db{c}{\torus, \Ls}$, we have
\begin{equation}
\R\Hom\pa{\E, \Vd \pa{p_*\C_{Z_k}}}\simeq \R\Hom\pa{p_*\C_{Z_k}, \Vd(\E)}. \label{vanishing}
\end{equation}
We will show that this cohomology vanishes for $1\leq k\leq n-1$. 

Let us first reduce the vanishing of (\ref{vanishing}) to (2) of Lemma \ref{sumup}. We define
\begin{equation}
F:=\set{m=\sum_{i=1}^na_ie_i^\vee\in M_\R\relmid -1\leq a_i<0}\bs M
\end{equation}
and $\hat{Z}_k:=Z_k\cap F$, then 
\begin{equation}
Z_k\bs\hat{Z}_k=
\bigsqcup_{\substack{1\leq j\leq k\\l_1+\cdots+l_j=-k\\ 
I:=\set{i_1,...,i_j}\subset\set{1,...,n}}}
\set{m=\sum_{i=1}^na_ie_i^\vee\in Z_k\relmid 
\substack{l_p\leq a_{p}<l_p+1\text{ for any } p\in I, \\
-1\leq a_q<0 \text{ for any } q\in\set{1,...,k}\bs I}
}.
\end{equation}
By induction, we only have to show 
\begin{equation}
\R\Hom(p_*\C_{\hat{Z}_k}, \Vd\pa{\E})\simeq 0
\end{equation}
to prove the vanishing of (\ref{vanishing}). Since there exists an exact triangle for $k\geq 2$
\begin{equation}
\C_{\tilde{Z}_k}\ra \C_{\hat{Z}_k}\ra \C_{{Z}_{k-1}}\xrightarrow{[1]} \C_{\tilde{Z}_k}[1]
\end{equation}
where $\tilde{Z}_k:=\hat{Z}_k-\hat{Z}_{k-1}$, we can further reduce the vanishing to
\begin{equation}\label{vanishing2}
\R\Hom\pa{p_*\C_{\tilde{Z}_k},\Vd\pa{\E}}\simeq \RGL{p\pa{\tilde{Z}_k}}{\torus, \Vd\E}\simeq 0
\end{equation}
where we define $\tilde{Z}_1:=\hat{Z}_1$.

\begin{figure}
\begin{center}
\begin{tabular}{ccc}
\begin{minipage}{0.3\hsize}
\begin{tikzpicture}
\draw (-2,0) -- (0,-2);
\draw[dashed] (0,-2) -- (-2,-2);
\draw[dashed] (-2,0) -- (-2,-2);
\draw[dashed] (0,0) -- (0,-2);
\draw[dashed] (0,0) -- (-2,0);
\fill[gray, opacity=.2] (0,-2) -- (-2,0) -- (0,0) --cycle;
\node at (-0.6,-0.6) {$\tilde{Z}_1=Z$};
\node[above] at (0,0) {$O$};
\end{tikzpicture}
\end{minipage}

\begin{minipage}{0.3\hsize}
\begin{tikzpicture}
\draw[dashed] (0,-2) -- (-2,-2);
\draw[dashed] (-2,0) -- (-2,-2);
\draw[dashed] (0,0) -- (0,-2);
\draw[dashed] (0,0) -- (-2,0);
\fill[gray, opacity=.2] (0,-2) --(-2,-2)-- (-2,0) -- (0,0) --cycle;
\node at (-1,-1) {$U_1$};
\node[above] at (0,0) {$[O]$};
\end{tikzpicture}
\end{minipage}

\begin{minipage}{0.3\hsize}
\begin{tikzpicture}
\draw[dashed] (-2,0) -- (0,-2);
\draw[dashed] (0,-2) -- (-2,-2);
\draw[dashed] (-2,0) -- (-2,-2);
\draw[dashed] (0,0) -- (0,-2);
\draw[dashed] (0,0) -- (-2,0);
\fill[gray, opacity=.2] (0,-2) -- (-2,0) -- (-2,-2) --cycle;
\node at (-1.5,-1.5) {$U_2$};
\node[above] at (0,0) {$[O]$};
\end{tikzpicture}
\end{minipage}
\end{tabular}
\caption{$\tilde{Z}_1$, $U_1$, and $U_2$ in $n=2$}
\end{center}
\end{figure}
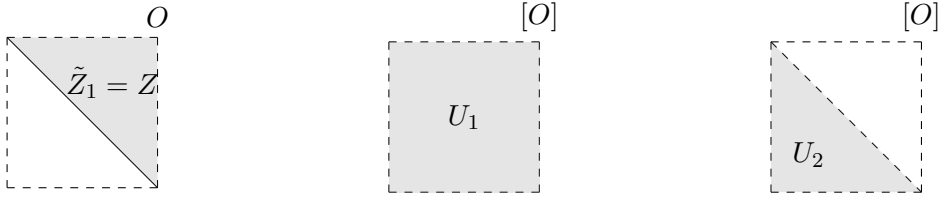

\begin{figure}
\begin{center}
\begin{tabular}{cc}
\begin{minipage}{0.4\hsize}
\begin{tikzpicture}[x=0.5cm,y=0.5cm,z=0.3cm, domain=-6:0]
\draw[dashed] (0,0,0)--(-3,0,0);
\draw[dashed] (0,0,0)--(0,-3,0);
\draw[dashed] (0,0,0)--(0,0,-3);
\draw[dashed] (-3,-3,0)--(-3,0,0);
\draw[dashed] (-3,-3,0)--(0,-3,0);
\draw[dashed] (0,-3,-3)--(0,0,-3);
\draw[dashed] (-3,0,-3)--(-3,0,0);
\draw[dashed] (0,-3,-3)--(0,-3,0);
\draw[dashed] (-3,0,-3)--(0,0,-3);
\draw[dashed] (-3,0,-3)--(-3,-3,-3);
\draw[dashed] (-3,-3,0)--(-3,-3,-3);
\draw[dashed] (-3,-3,-3)--(0,-3,-3);
\fill[gray, opacity=.1](0,0,0)--(-3,0,0)--(-3,-3,0)--(0,-3,0)--cycle;
\fill[gray, opacity=.1](0,0,0)--(0,-3,0)--(0,-3,-3)--(0,0,-3)--cycle;
\fill[gray, opacity=.1](0,0,0)--(-3,0,0)--(-3,0,-3)--(0,0,-3)--cycle;
\fill[gray, opacity=.1](-3,0,0)--(-3,-3,0)--(-3,0,-3)--cycle;
\fill[gray, opacity=.1](0,-3,0)--(0,-3,-3)--(-3,-3,0)--cycle;
\fill[gray, opacity=.1](0,0,-3)--(-3,0,-3)--(0,-3,-3)--cycle;
\fill[gray, opacity=.1](-3,-3,0)--(0,-3,-3)--(-3,0,-3)--cycle;
\node[right] at (0,0,0) {$O$};

\node at (-1.5,-1.5,-1.5) {$\hat{Z}_2$};
\draw (-3,-3,0)--(0,-3,-3)--(-3,0,-3)--cycle;
\end{tikzpicture}
\end{minipage}

\begin{minipage}{0.2\hsize}
\begin{tikzpicture}[x=0.5cm,y=0.5cm,z=0.3cm, domain=-6:0]
\draw[dashed] (0,0,0)--(-3,0,0);
\draw[dashed] (0,0,0)--(0,-3,0);
\draw[dashed] (0,0,0)--(0,0,-3);
\draw[dashed] (-3,-3,0)--(-3,0,0);
\draw[dashed] (-3,-3,0)--(0,-3,0);
\draw[dashed] (0,-3,-3)--(0,0,-3);
\draw[dashed] (-3,0,-3)--(-3,0,0);
\draw[dashed] (0,-3,-3)--(0,-3,0);
\draw[dashed] (-3,0,-3)--(0,0,-3);
\draw[dashed] (-3,0,-3)--(-3,-3,-3);
\draw[dashed] (-3,-3,0)--(-3,-3,-3);
\draw[dashed] (-3,-3,-3)--(0,-3,-3);
\fill[gray, opacity=.1](-3,0,0)--(-3,-3,0)--(0,-3,0)--cycle;
\fill[gray, opacity=.1](0,-3,0)--(0,-3,-3)--(0,0,-3)--cycle;
\fill[gray, opacity=.1](-3,0,0)--(-3,0,-3)--(0,0,-3)--cycle;
\fill[gray, opacity=.1](-3,0,0)--(-3,-3,0)--(-3,0,-3)--cycle;
\fill[gray, opacity=.1](0,-3,0)--(0,-3,-3)--(-3,-3,0)--cycle;
\fill[gray, opacity=.1](0,0,-3)--(-3,0,-3)--(0,-3,-3)--cycle;
\fill[gray, opacity=.1](-3,-3,0)--(0,-3,-3)--(-3,0,-3)--cycle;
\fill[gray, opacity=.1](-3,0,0)--(0,-3,0)--(0,0,-3)--cycle;

\node[right] at (0,0,0) {$O$};
\node at (-1.5,-1.5,-1.5) {$\tilde{Z}_2$};
\draw (-3,-3,0)--(0,-3,-3)--(-3,0,-3)--cycle;
\draw[dashed] (-3,0,0)--(0,-3,0)--(0,0,-3)--cycle;
\end{tikzpicture}
\end{minipage}
\end{tabular}
\caption{$\hat{Z}_2$ and $\tilde{Z}_2$ in $n$=3}
\end{center}
\end{figure}
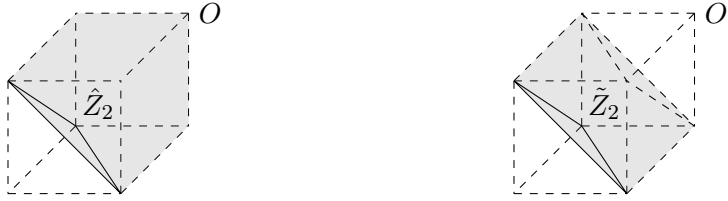

\begin{figure}
\begin{center}
\begin{tabular}{cc}
\begin{minipage}{0.5\hsize}
\begin{tikzpicture}[x=0.5cm,y=0.5cm,z=0.3cm, domain=-6:0]
\draw[dashed] (0,0,0)--(-3,0,0);
\draw[dashed] (0,0,0)--(0,-3,0);
\draw[dashed] (0,0,0)--(0,0,-3);
\draw[dashed] (-3,-3,0)--(-3,0,0);
\draw[dashed] (-3,-3,0)--(0,-3,0);
\draw[dashed] (0,-3,-3)--(0,0,-3);
\draw[dashed] (-3,0,-3)--(-3,0,0);
\draw[dashed] (0,-3,-3)--(0,-3,0);
\draw[dashed] (-3,0,-3)--(0,0,-3);
\draw[dashed] (-3,0,-3)--(-3,-3,-3);
\draw[dashed] (-3,-3,0)--(-3,-3,-3);
\draw[dashed] (-3,-3,-3)--(0,-3,-3);
\fill[gray, opacity=.1](-3,0,0)--(-3,-3,0)--(0,-3,0)--cycle;
\fill[gray, opacity=.1](0,-3,0)--(0,-3,-3)--(0,0,-3)--cycle;
\fill[gray, opacity=.1](-3,0,0)--(-3,0,-3)--(0,0,-3)--cycle;
\fill[gray, opacity=.1](-3,0,0)--(0,-3,0)--(0,0,-3)--cycle;
\fill[gray, opacity=.1](-3,0,0)--(-3,-3,0)--(-3,-3,-3)--(-3,0,-3)--cycle;
\fill[gray, opacity=.1](0,-3,0)--(0,-3,-3)--(-3,-3,-3)--(-3,-3,0)--cycle;
\fill[gray, opacity=.1](0,0,-3)--(-3,0,-3)--(-3,-3,-3)--(0,-3,-3)--cycle;

\node[right] at (0,0,0) {$[O]$};

\node at (-2,-2,-2) {$U_2$};
\draw[dashed] (-3,0,0)--(0,-3,0)--(0,0,-3)--cycle;
\end{tikzpicture}
\end{minipage}

\begin{minipage}{0.2\hsize}
\begin{tikzpicture}[x=0.5cm,y=0.5cm,z=0.3cm, domain=-6:0]
\draw[dashed] (0,0,0)--(-3,0,0);
\draw[dashed] (0,0,0)--(0,-3,0);
\draw[dashed] (0,0,0)--(0,0,-3);
\draw[dashed] (-3,-3,0)--(-3,0,0);
\draw[dashed] (-3,-3,0)--(0,-3,0);
\draw[dashed] (0,-3,-3)--(0,0,-3);
\draw[dashed] (-3,0,-3)--(-3,0,0);
\draw[dashed] (0,-3,-3)--(0,-3,0);
\draw[dashed] (-3,0,-3)--(0,0,-3);
\draw[dashed] (-3,0,-3)--(-3,-3,-3);
\draw[dashed] (-3,-3,0)--(-3,-3,-3);
\draw[dashed] (-3,-3,-3)--(0,-3,-3);
\fill[gray, opacity=.1](-3,0,-3)--(-3,-3,0)--(-3,-3,-3)--cycle;
\fill[gray, opacity=.1](-3,-3,0)--(0,-3,-3)--(-3,-3,-3)--cycle;
\fill[gray, opacity=.1](0,-3,-3)--(-3,0,-3)--(-3,-3,-3)--cycle;
\fill[gray, opacity=.1](-3,-3,0)--(0,-3,-3)--(-3,0,-3)--cycle;
\node[right] at (0,0,0) {[$O$]};

\node at (-2.5,-2.5,-2.5) {$U_3$};
\draw[dashed] (-3,-3,0)--(0,-3,-3)--(-3,0,-3)--cycle;
\end{tikzpicture}
\end{minipage}
\end{tabular}
\caption{$U_2$ and $U_3$ in $n$=3}
\end{center}
\end{figure}
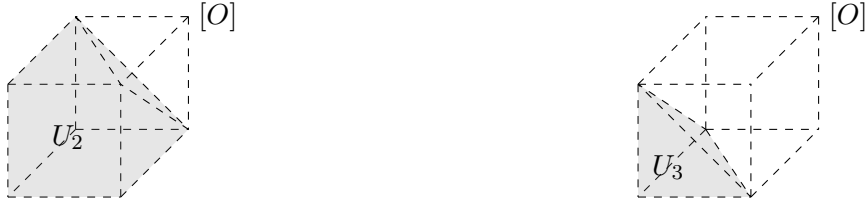

We define the subset $U_k$ of $\torus$ by
\begin{align}
U_{k}&:=p\pa{\tilde{U}_k}\\
\tilde{U}_k&:=F\bs \hat{Z}_{k-1},
\end{align} 
then $p\pa{\tilde{Z}_k}$ is a closed subset of $U_k$ (Figure 3, 4 and 5). Here we set $\hat{Z}_0=\varnothing$. In the following, we fix an integer $k$ with $1\leq k\leq n-1$. Then, we have the following exact triangle:
\begin{equation}\label{isomtriangle}
\RGL{p\pa{\tilde{Z}_k}}{U_k, \Vd\E}\ra \RG{U_k,\Vd\E}\ra \RG{U_{k+1}, \Vd\E}\xrightarrow{[1]}\RGL{p\pa{\tilde{Z}_k}}{U_k, \Vd\E}[1].
\end{equation}
Now the following lemma is clear:
\begin{lemma}\label{sumup}
For $\E\in \Db{c}{\torus,\Lss}$, the following are equivalent:
\begin{enumerate}[label={\upshape(\arabic*)}]
\item $\R\Hom(\E, \Vd(p_*\C_{Z_k}))\simeq 0$ for $1\leq k\leq n-1$.
\item $\RG{U_k,\Vd\E}\xrightarrow{\simeq} \RG{U_{k+1}, \Vd\E}$ for $1\leq k\leq n-1$.
\end{enumerate}
\end{lemma}

We will show the isomorphism (2) above in what follows. Let us define the family of open subsets $\mathcal{V}:=\set{V_s}_{(-\infty,1)}$ of $U_k$ (Figure 6 and 7) by 
\begin{equation}
V_s:=
\begin{cases}
U_{k+1} & \text{for}\hspace{1mm} s\in (-\infty, 0), \\
U_{k+1}\cup p\pa{\pa{\tilde{U}_k-\frac{1-s}{n}\cdot e_E}\cap F}& \text{for}\hspace{1mm} s\in [0, 1).
\end{cases}
\end{equation}
Then, $\mathcal{V}$ obviously satisfies (1) of Theorem \ref{nonchara}. 

For $s\leq t<1$, we have
\begin{equation}
V_t\bs V_s=
\begin{cases}
\varnothing \text{ or} \\
p\pa{\pa{\tilde{U}_k-\frac{1-s}{n}\cdot e_E}\cap \set{n\in N_\R\relmid \hull{n, e_E}\geq -k-t}\cap F}.
\end{cases}
\end{equation}
It follows that $V_t\bs V_s$ is relatively compact in $U_k$ i.e. $\mathcal{V}$ satisfies (2) of Theorem \ref{nonchara}.

For $s\leq t$, we have
\begin{equation}
\bigcap_{u>s}\Cl{V_u\bs V_s}\bs V_t=
\begin{cases}
p\pa{\pa{\partial\tilde{U}_k-\frac{1-s}{n}\cdot e_E}\cap F} &\text{if $s=t>0$},\\
\varnothing &\text{otherwise}.
\end{cases}
\end{equation}
Take $y\in p\pa{\pa{\partial\tilde{U}_k-\frac{1-s}{n}\cdot e_E}\cap F}$. We can see that there exists a neighbourhood $B(y)$ of $y$ such that $\gamma:=N_y\pa{U_k\bs V_s}$ and $U_k\bs V_s$ is canonically isomorphic in $B(y)$. The cone $\gamma$ is an $n$-dimensional closed polyhedral convex cone which is contained in $e_E^\vee$ under the canonical identification $T_y\torus\cong M_\R$. Take an $n$-dimensional proper closed polyhedral covex cone $\delta\subset N_\R$ such that $\sum_{i=1}^n\R_{\geq 0}\cdot e_i\bs\set{0}\subset \Int(\delta)$ and $\delta\subset \Int(e_E^\vee)$. We also have $\gamma\subset e_E^\vee$. Hence we have $-\delta\cap \gamma=\set{0}$. We can also see that $\delta^\vee\bs\set{0}\subset \sum_{i=1}^n\R_{>0}\cdot e_i^\vee$ and  
\begin{equation}
\begin{split}
\gamma^\vee =N^*_y(U_k\bs V_s)&\subset N^*_y\pa{p\pa{\tilde{U}_k-\frac{1-s}{n}\cdot e_E}}+ N_y^*\pa{p(F)}\\
&\subset \sum_{i=1}^n \R_{\geq 0}\cdot e_i.
\end{split}
\end{equation}
Hence we have $\Relint(\gamma^\vee)+\delta^\vee\subset \sum_{i=1}^n\R_{>0}\cdot e_i$.

Since $\ms(\Vd\E)=-\ms(\E)\subset -\Ls$, we have $U_k\times \pa{\bigcup_{i=1}^n\R_{>0}e_i}\cap \ms(\Vd\E)=\varnothing$. Hence, by Lemma \ref{edgevanishing}, we obtain $(\RGL{U_k\bs V_s}{\Vd\E})_{y}\simeq 0$. Then we can use Theorem \ref{nonchara} for $\mathcal{V}$ and we know the middle arrow of (\ref{isomtriangle}) is an isomorphism. By Lemma \ref{sumup}, we conclude that $\E\in ^{\perp}\hull{\Kss\pa{\D_{-n+1}},..., \Kss\pa{\D_{-1}}}$.

\subsection*{Step 2 ($\supset$)} Conversely, take $\E\in \lperp\hull{\Kss\pa{\D_{-r+1}}, ..., \Kss\pa{\D_{-1}}}$. We have to show that $\ms(\E)\subset \Ls$, or equivalently,
\begin{equation}\label{intersect}
\ms(\E)\cap \bigcup_{\substack{\rho_E\subset\sigma\in \HSigma\\ \dim\sigma<n}}p\pa{\Relint\pa{\pa{\sigma^\perp+M}\cap F}}\times(-\sigma)=\varnothing.
\end{equation}
For a proof by contradiction, we assume that there exists an element $([x_0],-\xi_0)$ in the LHS of the above (\ref{intersect}). Since $\ms(\E)$ is a conic Lagrangian subset \cite{KS}, we can assume that $\xi_0\in \Relint(\sigma)$ for some $\sigma\in \HSigma$ such that $\rho_E\subset \sigma$ and $\dim\sigma<n$. Fix $k\in \set{1,.., n-1}$, then we can take the unique lift $x_0=\sum_{i=1}^na_ie_i^\vee \in M_\R$ satisfying $-1\leq a_i<0$, $\sum_{i=1}^na_i=-k$.

We can see that there exists a neighbourhood $B([x_0])$ of $[x_0]$ such that $\tgamma:=N_{[x_0]}\pa{U_{k}\bs U_{k+1}}$ is canonically isomorphic to $U_k\bs U_{k+1}$ in $B([x_0])$. Moreover, via the canonical identification $T_{[x_0]}\torus\cong M_\R$, we have $\tgamma\cong \sigma^\vee$. On the other hand, we can take an $n$-dimensional proper polyhedral closed convex cone $\tdelta$ such that $\sum_{i=1}^n\R_{\geq0}\cdot e_i^\vee\bs\set{0}\subset \Int\pa{\tdelta}$, $\tdelta^\vee\cap\pa{\tgamma^\vee\bs\Relint\pa{\tgamma^\vee}}=\set{0}$, and $\xi_0\in \Int\pa{\tdelta^\vee}$.

The cones $\tgamma$ and $\tdelta$ clearly satisfy (ii)-(iv) of Lemma \ref{nonvanishing}. 
By the definitions of $\tgamma$ and $\tdelta$, $\pa{\tdelta^\vee+\tgamma^\vee}\subset \sum_{i=1}^n\R_{\geq 0}\cdot e_i^\vee$.
Since $\ms(\E)\subset \Lss$, we have 
\begin{equation}
\begin{split}
\ms(\Vd\E)\cap \pa{B([x_0])\bs \tgamma}\times \pa{\tdelta^\vee+\tgamma^\vee}&\subset -\Lss\cap \pa{B([x_0])\bs\tgamma}\times \pa{\sum_{i=1}^n\R_{\geq 0}e_i}\\
&\subset \bigcup_{0\neq\tau\prec\sigma}\tau^\perp\times \tau\cup\pa{B([x_0])\bs\tgamma}\times\set{0}\\
&=\bigsqcup_{0\neq\tau\prec\sigma=\tgamma^\vee}R_\tau\times \tau\cup\pa{B([x_0])\bs\tgamma}\times\set{0}.
\end{split}
\end{equation}
In the second line, we use the canonical identification $B([x_0])\hookrightarrow T_{[x_0]}\torus\cong M_\R$. This shows that $\tdelta$ and $\tgamma$ also satisfy \ref{281} of Lemma \ref{nonvanishing}.
 
By Lemma \ref{nonvanishing}, there exists $x$ near $x_0$ such that
\begin{equation}
\RGL{U_{k}\bs U_{k+1}}{p\pa{\pa{x-\tdelta}\cap \tilde{U}_k\bs\tilde{U}_{k+1}}, \Vd\E}\not\simeq 0.
\end{equation}
We define $S:=p\pa{\tilde{U}_k\bs \tilde{U}_{k+1}\cap \pa{x-\tdelta}}$. Then we have
\begin{equation}\label{notsim}
\RG{U_{k+1}\cup S,\Vd\E}\not\simeq \RG{U_{k+1},\Vd\E}.
\end{equation}

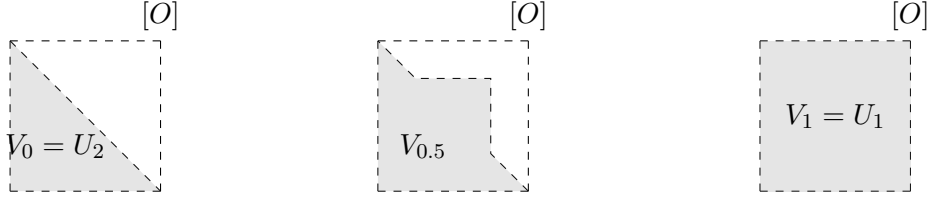
\begin{figure}
\begin{center}
\begin{tabular}{ccc}
\begin{minipage}{0.3\hsize}
\begin{tikzpicture}
\draw[dashed] (-2,0) -- (0,-2);
\draw[dashed] (0,-2) -- (-2,-2);
\draw[dashed] (-2,0) -- (-2,-2);
\draw[dashed] (0,0) -- (0,-2);
\draw[dashed] (0,0) -- (-2,0);
\fill[gray, opacity=.2] (0,-2) -- (-2,0) -- (-2,-2) --cycle;
\node at (-1.4,-1.4) {$V_0=U_2$};
\node[above] at (0,0) {$[O]$};
\end{tikzpicture}

\end{minipage}

\begin{minipage}{0.3\hsize}
\begin{tikzpicture}
\draw[dashed] (0,-2) -- (-0.5,-1.5)--(-0.5,-0.5)-- (-1.5, -0.5)--(-2,0);
\draw[dashed] (0,-2) -- (-2,-2);
\draw[dashed] (-2,0) -- (-2,-2);
\draw[dashed] (0,0) -- (0,-2);
\draw[dashed] (0,0) -- (-2,0);
\fill[gray, opacity=.2] (0,-2) -- (-0.5,-1.5)--(-0.5,-0.5)-- (-1.5, -0.5)--(-2,0)--(-2,-2)--cycle;
\node at (-1.4,-1.4) {$V_{0.5}$};
\node[above] at (0,0) {$[O]$};
\end{tikzpicture}
\end{minipage}

\begin{minipage}{0.3\hsize}
\begin{tikzpicture}
\draw[dashed] (0,-2) -- (-2,-2);
\draw[dashed] (-2,0) -- (-2,-2);
\draw[dashed] (0,0) -- (0,-2);
\draw[dashed] (0,0) -- (-2,0);
\fill[gray, opacity=.2] (0,-2) --(-2,-2)-- (-2,0) -- (0,0) --cycle;
\node at (-1,-1) {$V_1=U_1$};
\node[above] at (0,0) {$[O]$};
\end{tikzpicture}
\end{minipage}
\end{tabular}
\caption{$V_0$, $V_{0.5}$, and $V_1$ in $n=2$}
\end{center}
\end{figure}

\begin{figure}
\begin{center}
\begin{tabular}{cc}
\begin{minipage}{0.3\hsize}
\begin{tikzpicture}[x=0.5cm,y=0.5cm,z=0.3cm, domain=-6:0]
\draw[dashed] (0,0,0)--(-3,0,0);
\draw[dashed] (0,0,0)--(0,-3,0);
\draw[dashed] (0,0,0)--(0,0,-3);
\draw[dashed] (-3,-3,0)--(-3,0,0);
\draw[dashed] (-3,-3,0)--(0,-3,0);
\draw[dashed] (0,-3,-3)--(0,0,-3);
\draw[dashed] (-3,0,-3)--(-3,0,0);
\draw[dashed] (0,-3,-3)--(0,-3,0);
\draw[dashed] (-3,0,-3)--(0,0,-3);
\draw[dashed] (-3,0,-3)--(-3,-3,-3);
\draw[dashed] (-3,-3,0)--(-3,-3,-3);
\draw[dashed] (-3,-3,-3)--(0,-3,-3);
\fill[gray, opacity=.1](-3,0,-3)--(-3,-3,0)--(-3,-3,-3)--cycle;
\fill[gray, opacity=.1](-3,-3,0)--(0,-3,-3)--(-3,-3,-3)--cycle;
\fill[gray, opacity=.1](0,-3,-3)--(-3,0,-3)--(-3,-3,-3)--cycle;
\fill[gray, opacity=.1](-3,-3,0)--(0,-3,-3)--(-3,0,-3)--cycle;
\node[right] at (0,0,0) {[$O$]};

\node at (-2.5,-2.5,-2.5) {$V_0=U_3$};
\draw[dashed] (-3,-3,0)--(0,-3,-3)--(-3,0,-3)--cycle;
\end{tikzpicture}
\end{minipage}

\begin{minipage}{0.3\hsize}
\begin{tikzpicture}[x=0.5cm,y=0.5cm,z=0.3cm, domain=-6:0]
\draw[dashed] (0,0,0)--(-3,0,0);
\draw[dashed] (0,0,0)--(0,-3,0);
\draw[dashed] (0,0,0)--(0,0,-3);
\draw[dashed] (-3,-3,0)--(-3,0,0);
\draw[dashed] (-3,-3,0)--(0,-3,0);
\draw[dashed] (0,-3,-3)--(0,0,-3);
\draw[dashed] (-3,0,-3)--(-3,0,0);
\draw[dashed] (0,-3,-3)--(0,-3,0);
\draw[dashed] (-3,0,-3)--(0,0,-3);
\draw[dashed] (-3,0,-3)--(-3,-3,-3);
\draw[dashed] (-3,-3,0)--(-3,-3,-3);
\draw[dashed] (-3,-3,-3)--(0,-3,-3);
\fill[gray, opacity=.1](-3,0,-3)--(-3,-1,-2)--(-3,-1,-1.5)--(-3,-1.5,-1)--(-3,-2,-1)--(-3,-3,0)--(-3,-3,-3)--cycle;
\fill[gray, opacity=.1](-3,-3,0)--(-2,-3,-1)--(-1.5,-3,-1)--(-1,-3,-1.5)--(-1,-3,-2)--(0,-3,-3)--(-3,-3,-3)--cycle;
\fill[gray, opacity=.1](0,-3,-3)--(-1,-2,-3)--(-1,-1.5,-3)--(-1.5,-1,-3)--(-2,-1,-3)--(-3,0,-3)--(-3,-3,-3)--cycle;
\fill[gray, opacity=.1](-3,-3,0)--(-3,-2,-1)--(-2,-3,-1)--cycle;
\fill[gray, opacity=.1](-3,0,-3)--(-3,-1,-2)--(-2,-1,-3)--cycle;
\fill[gray, opacity=.1](0,-3,-3)--(-1,-3,-2)--(-1,-2,-3)--cycle;
\fill[gray, opacity=.1](-3,-2,-1)--(-2,-3,-1)--(-1.5,-3,-1)--(-3,-1.5,-1)--cycle;
\fill[gray, opacity=.1](-3,-1,-2)--(-2,-1,-3)--(-1.5,-1,-3)--(-3,-1,-1.5)--cycle;
\fill[gray, opacity=.1](-1,-2,-3)--(-1,-3,-2)--(-1,-3,-1.5)--(-1,-1.5,-3)--cycle;
\fill[gray, opacity=.1](-1,-3,-1.5)--(-1.5,-3,-1)--(-3,-1.5,-1)--(-3,-1,-1.5)--(-1.5,-1,-3)--(-1,-1.5,-3)--cycle;
\node[right] at (0,0,0) {[$O$]};

\node at (-2.5,-2.5,-2.5) {$V_{0.5}$};
\draw[dashed] (-3,-3,0)--(-3,-2,-1)--(-3,-1.5,-1)--(-3,-1,-1.5)--(-3,-1,-2)--(-3,0,-3)--(-2,-1,-3)--(-1.5,-1,-3)
--(-1,-1.5,-3)--(-1,-2,-3)--(0,-3,-3)--(-1,-3,-2)--(-1,-3,-1.5)--(-1.5,-3,-1)--(-2,-3,-1)--(-3,-3,0);
\draw[dashed](-3,-2,-1)--(-3,-1.5,-1)--(-1.5,-3,-1)--(-2,-3,-1);
\draw[dashed](-2,-3,-1)--(-1.5,-3,-1)--(-1,-3,-1.5)--(-1,-3,-2);
\draw[dashed](-1,-3,-2)--(-1,-3,-1.5)--(-1,-1.5,-3)--(-1,-2,-3);
\draw[dashed](-1,-2,-3)--(-1,-1.5,-3)--(-1.5,-1,-3)--(-2,-1,-3);
\draw[dashed](-2,-1,-3)--(-1.5,-1,-3)--(-3,-1,-1.5)--(-3,-1,-2);
\draw[dashed](-3,-1,-2)--(-3,-1,-1.5)--(-3,-1.5,-1)--(-3,-2,-1);
\end{tikzpicture}
\end{minipage}

\begin{minipage}{0.3\hsize}
\begin{tikzpicture}[x=0.5cm,y=0.5cm,z=0.3cm, domain=-6:0]
\draw[dashed] (0,0,0)--(-3,0,0);
\draw[dashed] (0,0,0)--(0,-3,0);
\draw[dashed] (0,0,0)--(0,0,-3);
\draw[dashed] (-3,-3,0)--(-3,0,0);
\draw[dashed] (-3,-3,0)--(0,-3,0);
\draw[dashed] (0,-3,-3)--(0,0,-3);
\draw[dashed] (-3,0,-3)--(-3,0,0);
\draw[dashed] (0,-3,-3)--(0,-3,0);
\draw[dashed] (-3,0,-3)--(0,0,-3);
\draw[dashed] (-3,0,-3)--(-3,-3,-3);
\draw[dashed] (-3,-3,0)--(-3,-3,-3);
\draw[dashed] (-3,-3,-3)--(0,-3,-3);
\fill[gray, opacity=.1](-3,0,0)--(-3,-3,0)--(0,-3,0)--cycle;
\fill[gray, opacity=.1](0,-3,0)--(0,-3,-3)--(0,0,-3)--cycle;
\fill[gray, opacity=.1](-3,0,0)--(-3,0,-3)--(0,0,-3)--cycle;
\fill[gray, opacity=.1](-3,0,0)--(0,-3,0)--(0,0,-3)--cycle;
\fill[gray, opacity=.1](-3,0,0)--(-3,-3,0)--(-3,-3,-3)--(-3,0,-3)--cycle;
\fill[gray, opacity=.1](0,-3,0)--(0,-3,-3)--(-3,-3,-3)--(-3,-3,0)--cycle;
\fill[gray, opacity=.1](0,0,-3)--(-3,0,-3)--(-3,-3,-3)--(0,-3,-3)--cycle;

\node[right] at (0,0,0) {$[O]$};

\node at (-2,-2,-2) {$V_1=U_2$};
\draw[dashed] (-3,0,0)--(0,-3,0)--(0,0,-3)--cycle;
\end{tikzpicture}
\end{minipage}

\end{tabular}
\caption{$V_0$, $V_{0.5}$ and $V_1$ in $n$=3}
\end{center}
\end{figure}
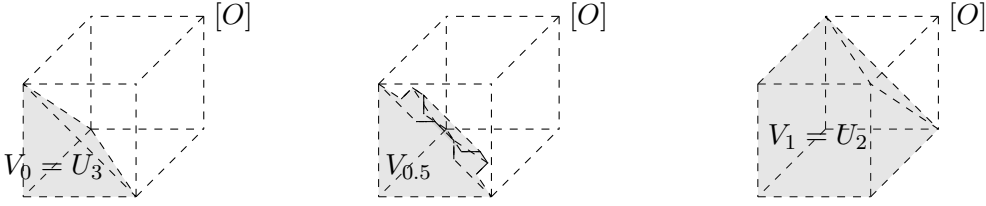

We define a family of open subsets $\mathcal{V}':=\{V'_s\}_{s\in(-\infty,1)}$ in $U_k$ (Figure 8) by
\begin{equation}
V'_s:=
\begin{cases}
U_{k+1} \cup S& \text{for}\hspace{1mm} s\in (-\infty, 0),\\
U_{k+1}\cup S\cup p\pa{F\cap\pa{s\cdot\pa{\tilde{U}_k-x}+x}}& \text{for}\hspace{1mm} s\in [0, 1).
\end{cases}
\end{equation}
We can see that the family $\mathcal{V}'$ satisfies (1) and (2) of Theorem \ref{nonchara}. 

For $s\leq t$, we have
\begin{equation}
\bigcap_{u>s}\Cl{V'_u\bs V'_s}\bs V'_t
=\begin{cases}
W_1\cup W_2\cup W_3 &\text{$s=t\geq 0$,}\\
\varnothing & \text{otherwise}.
\end{cases}
\end{equation}
where
\begin{align}
W_1&:=p\pa{F\cap\pa{s\cdot\pa{\partial\tilde{U}_k-x}+x}}\bs\pa{\Cl{U_{k+1}}\cup S},\\
W_2&:=\partial S \cap p\pa{F\cap\pa{s\cdot\pa{\partial\tilde{U}_k-x}+x}}\bs W_3,\\
W_3&:=p\pa{\set{m\in M_\R\relmid \hull{m,e_E}=-k-1}\cap\pa{s\cdot\pa{\partial\tilde{U}_k-x}+x}}.
\end{align}
Some examples of $W_i$ are depicted in Figure 9.
We take $y\in \bigcap_{u>s}\Cl{V'_u\bs V'_s}\bs V'_s$ for $0\leq s<1$ and define a cone $\hgamma$ as $\hgamma:=N_y\pa{U_k\bs V'_s}$. There exists a neighbourhood $B(y')$ of $y'$ such that $\gamma$ is canonically isomorphic to $U_k\bs V_s'$ in $B(y')$.

If $y\in W_2$, we have $\hgamma^\vee \subset \sum_{i=1}^n\R_{\geq 0}\cdot e_i$ via the canonical identification $T_y^*\torus \cong N_\R$. Since $\Vd\E$ does not have its microsupport in its first quadrant at $y$, we have $\pa{\RGL{U_k\bs V_s'}{\Vd\E}}_y\simeq 0$ by Lemma \ref{cone2}.

If $y\in W_1$, we have $\hgamma^\vee\subset \sum_{i=0}\R_{\geq 0}\cdot e_i$. On the other hand, the microsupport of $\Vd\E$ at $y$ is contained in a face of $\sum_{i=1}^n\R_{\geq 0}\cdot e_i$. Hence we have $\pa{\RGL{U_k\bs V_s'}{\Vd\E}}_y\simeq 0$ by Lemma \ref{edgevanishing}.

If $y\in W_3$, $\hgamma$ is contained in some $\sigma\in \HSigma$ such that $\rho_E\subset \sigma$. Hence, again by Lemma \ref{edgevanishing}, we have $\pa{\RGL{U_k\bs V_s'}{\Vd\E}}_y\simeq 0$.

%

Then, by using Theorem \ref{nonchara} for $\mathcal{V}'$, we have
\begin{equation}\label{sim}
\RG{U_k,\Vd\E}\simeq \RG{U_{k+1}\cup S,\Vd\E}.
\end{equation}

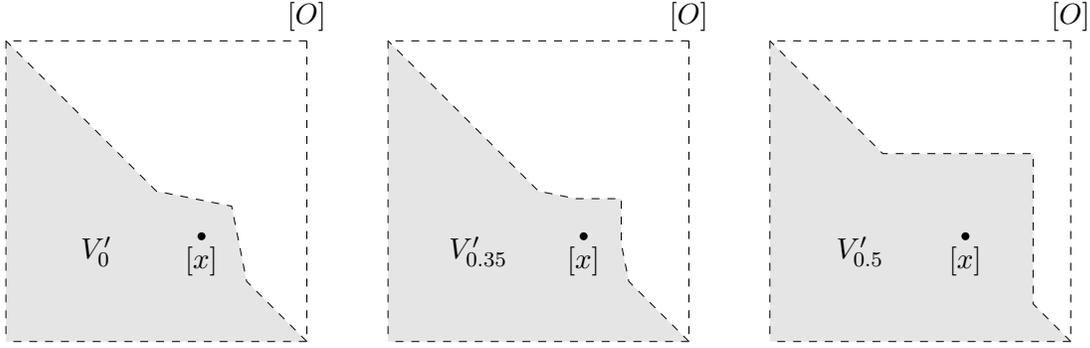
\begin{figure}
\begin{center}
\begin{tabular}{ccc}
\begin{minipage}{0.3\hsize}
\begin{tikzpicture}
\draw[dashed] (0,-4) -- (-4,-4);
\draw[dashed] (-4,0) -- (-4,-4);
\draw[dashed] (0,0) -- (0,-4);
\draw[dashed] (0,0) -- (-4,0);
\draw[dashed](0,-4) -- (-0.8,-3.2)--(-1.0,-2.2)-- (-2.0, -2.0)--(-4,0);
\fill[gray, opacity=.2] (0,-4) -- (-0.8,-3.2)--(-1,-2.2)-- (-2, -2)--(-4,0)--(-4,-4)--cycle;
\fill (-1.4,-2.6) circle(1.5pt);
\node[below] at (-1.4,-2.6) {$[x]$}; 
\node at (-2.8,-2.8) {$V_0'$};
\node[above] at (0,0) {$[O]$};
\end{tikzpicture}
\end{minipage}

\begin{minipage}{0.3\hsize}
\begin{tikzpicture}
\draw[dashed] (0,-4) -- (-4,-4);
\draw[dashed] (-4,0) -- (-4,-4);
\draw[dashed] (0,0) -- (0,-4);
\draw[dashed] (0,0) -- (-4,0);
\draw[dashed] (0,-4) -- (-0.8,-3.2)--(-0.9,-2.7)--(-0.9,-2.1)--(-1.5,-2.1)--(-2, -2)--(-4,0);
\fill[gray, opacity=.2] (0,-4) -- (-0.8,-3.2)--(-0.9,-2.7)--(-0.9,-2.1)--(-1.5,-2.1)--(-2, -2)--(-4,0)--(-4,-4)--cycle;
\fill (-1.4,-2.6) circle(1.5pt);
\node[below] at (-1.4,-2.6) {$[x]$}; 
\node at (-2.8,-2.8) {$V_{0.35}'$};
\node[above] at (0,0) {$[O]$};
\end{tikzpicture}
\end{minipage}

\begin{minipage}{0.3\hsize}
\begin{tikzpicture}
\draw[dashed] (0,-4) -- (-0.5,-3.5)--(-0.5,-1.5)-- (-2.5, -1.5)--(-4,0);
\draw[dashed] (0,-4) -- (-4,-4);
\draw[dashed] (-4,0) -- (-4,-4);
\draw[dashed] (0,0) -- (0,-4);
\draw[dashed] (0,0) -- (-4,0);
\fill[gray, opacity=.2] (0,-4) -- (-0.5,-3.5)--(-0.5,-1.5)-- (-2.5, -1.5)--(-4,0)--(-4,-4)--cycle;
\fill (-1.4,-2.6) circle(1.5pt);
\node[below] at (-1.4,-2.6) {$[x]$}; 
\node at (-2.8,-2.8) {$V_{0.5}'$};
\node[above] at (0,0) {$[O]$};
\end{tikzpicture}
\end{minipage}

\end{tabular}
\caption{$V'_0$, $V'_{0.35}$, and $V'_{0.5}$ in $n=2$}
\end{center}
\end{figure}

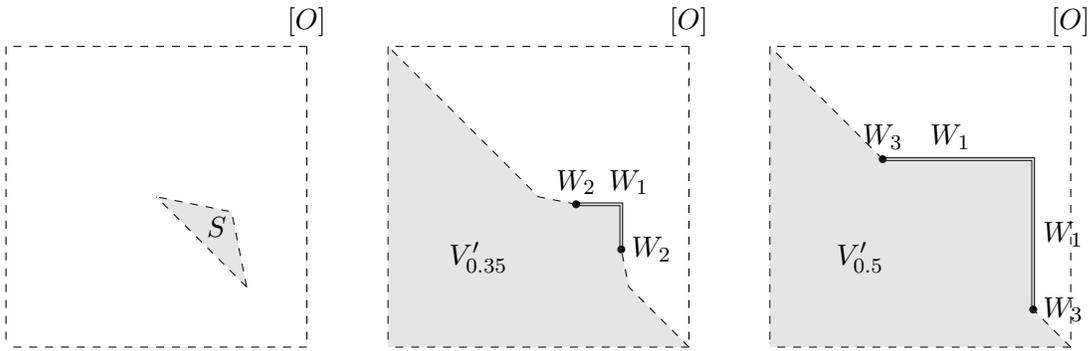
\begin{figure}
\begin{center}
\begin{tabular}{ccc}
\begin{minipage}{0.3\hsize}
\begin{tikzpicture}
\draw[dashed] (0,-4) -- (-4,-4);
\draw[dashed] (-4,0) -- (-4,-4);
\draw[dashed] (0,0) -- (0,-4);
\draw[dashed] (0,0) -- (-4,0);
\draw[dashed](-0.8,-3.2)--(-1.0,-2.2)-- (-2.0, -2.0);
\draw[dashed](-0.8,-3.2)--(-2.0, -2.0);
\fill[gray, opacity=.2](-0.8,-3.2)--(-1,-2.2)-- (-2, -2)--cycle;
\node at (-1.2,-2.4) {$S$};
\node[above] at (0,0) {$[O]$};
\end{tikzpicture}
\end{minipage}

\begin{minipage}{0.3\hsize}
\begin{tikzpicture}
\draw[dashed] (0,-4) -- (-4,-4);
\draw[dashed] (-4,0) -- (-4,-4);
\draw[dashed] (0,0) -- (0,-4);
\draw[dashed] (0,0) -- (-4,0);
\draw[double] (-0.9,-2.7)--(-0.9,-2.1)--(-1.5,-2.1);
\fill (-0.9,-2.7) circle(1.5pt);
\fill (-1.5,-2.1) circle(1.5pt);
\draw[dashed] (0,-4) -- (-0.8,-3.2)--(-0.9,-2.7);
\draw[dashed] (-1.5,-2.1)--(-2, -2)--(-4,0);
\fill[gray, opacity=.2] (0,-4) -- (-0.8,-3.2)--(-0.9,-2.7)--(-0.9,-2.1)--(-1.5,-2.1)--(-2, -2)--(-4,0)--(-4,-4)--cycle;
\node[above] at (-0.8,-2.1) {$W_1$};
\node[above] at (-1.5,-2.1) {$W_2$};
\node[right] at (-0.9,-2.7) {$W_2$};
\node at (-2.8,-2.8) {$V_{0.35}'$};
\node[above] at (0,0) {$[O]$};
\end{tikzpicture}
\end{minipage}

\begin{minipage}{0.3\hsize}
\begin{tikzpicture}
\draw[double] (-0.5,-3.5)--(-0.5,-1.5)-- (-2.5, -1.5);
\fill (-0.5,-3.5) circle(1.5pt);
\fill (-2.5,-1.5) circle(1.5pt);
\draw[dashed] (-2.5, -1.5)--(-4,0);
\draw[dashed] (0,-4) -- (-0.5,-3.5);
\draw[dashed] (0,-4) -- (-4,-4);
\draw[dashed] (-4,0) -- (-4,-4);
\draw[dashed] (0,0) -- (0,-4);
\draw[dashed] (0,0) -- (-4,0);
\fill[gray, opacity=.2] (0,-4) -- (-0.5,-3.5)--(-0.5,-1.5)-- (-2.5, -1.5)--(-4,0)--(-4,-4)--cycle;
\node[above] at (-1.6,-1.5) {$W_1$};
\node[right] at (-0.5,-2.5) {$W_1$};
\node[right] at (-0.5,-3.5) {$W_3$};
\node[above] at (-2.5,-1.5) {$W_3$};
\node at (-2.8,-2.8) {$V_{0.5}'$};
\node[above] at (0,0) {$[O]$};
\end{tikzpicture}
\end{minipage}

\end{tabular}
\caption{$S$, $W_i$ in $n=2$}
\end{center}
\end{figure}

As a consequence, we have the following commutative diagram:
\begin{equation}\label{diagram}
\xymatrix{\ar@{}[rrd]|{\circlearrowright}
\RG{U_{k+1}\cup S, \Vd\E}\ar[rr]^{\not\sim} &&\RG{U_{k+1},\Vd\E}\\
&\RG{U_k,\Vd\E}\ar[lu]^\sim\ar[ru]_{\sim}&}
\end{equation}
Since all arrows in (\ref{diagram}) are restriction morphisms, this diagram is commutative.
The upper arrow is (\ref{notsim}) and the right arrow is the isomorphism (\ref{sim}). The left arrow is also an isomorphism by the assumption and Lemma \ref{sumup}. This diagram is obviously absurd. Hence, $\pa{[x_0],\xi_0}\not\in \ms(\E)$. This completes the proof.
\end{proof}


\section{Proofs of the main theorems}

\begin{proof}[Proof of Theorem \ref{iff}]
First, suppose that $\Kss\colon\Db{}{\coh\Xss}\rightarrow \Db{c}{\torus, \Lss}$ is an
equivalence. Then, semi-orthogonal decomposition described in Orlov's theorem (\ref{orlov}) implies
\begin{equation}
^\perp\hull{\Kss\pa{\OX_E\pa{kE}}}_{1\leq k\leq n-1}\cong\Kss\pa{\pi^*\Db{}{\coh\Xs}}.
\end{equation}
On the other hand, Theorem \ref{sod} implies
\begin{equation}
^\perp\hull{\Kss\pa{\OX_E\pa{kE}}}_{1\leq k\leq n-1}\cong\iota\pa{\Db{c}{\torus, \Lss}}.
\end{equation}
Hence,
\begin{equation}
\begin{split}
\Ks\pa{\Db{}{\coh\Xs}}&\cong \Kss\pa{\pi^*\Db{}{\coh \Xs}} \\
&\cong ^\perp\hspace{-1mm}\hull{\Kss\pa{\OX_E(kE)}}_{1\leq k\leq n-1}\\
&\cong \iota\pa{\Db{c}{\torus, \Ls}}\\
&\cong \Db{c}{\torus, \Ls}
\end{split}
\end{equation}
where the equivalence in the first line is Theorem \ref{kappa} (1). 

Conversely, we assume that Conjecture \ref{conjecture} holds for $\Sigma$. Take $\E\in
\kappa_{\hat{\Sigma}}(\D^\ueb(\coh\Xss))^\perp$, then $\E\in \kappa_{\hat{\Sigma}}(\mathcal{O}_E(kE))^\perp$ for $1\leq k\leq n-1$ and $\E \in \kappa_{\hat{\Sigma}}(\pi^*\D^\ueb\coh(\Xs)))\cong \iota(\D^\ueb(\torus,
\overline{\Lambda_{\Sigma}}))^\perp$. Hence, by Theorem \ref{sod}, $\E\cong 0$. This completes the proof.
\end{proof}

Using Theorem \ref{iff}, we have a proof of Conjecture \ref{conjecture} for smooth complete toric surfaces.

\begin{proof}[Proof of Theorem \ref{intromain}]
We use toric minimal model program for toric surfaces. For any toric surfaces, after some toric blow-downs, we have
$\mathbb{P}^2$, $\mathbb{P}^1\times \mathbb{P}^1$ or Hirzebruch surfaces $\mathbb{F}_n$. By Theorem $\ref{iff}$, it is
enough to show that Conjecture \ref{conjecture} holds for these toric minimal models. Moreover, $\mathbb{P}^2$ and $\mathbb{F}_n$
are obtained by successive blow-ups and blow-downs of $\mathbb{P}^1\times \mathbb{P}^1$. Hence, all cases may be reduced to the Conjecture \ref{conjecture} for
$\mathbb{P}^1\times \mathbb{P}^1$. The last case has already been shown by Treumann (Theorem \ref{known} (1)). This completes the proof.
\end{proof}

Finally, we obtain Corollary \ref{corhms} by combining Theorem \ref{intromain} with Theorem \ref{hms}.

\bibliographystyle{alpha}
\bibliography{blowupinccc.bib}

\begin{bibdiv}
\begin{biblist}

\bib{BK}{article}{
      author={Bondal, A.~I.},
      author={Kapranov, M.~M.},
       title={Representable functors, {S}erre functors, and reconstructions},
        date={1989},
        ISSN={0373-2436},
     journal={Izv. Akad. Nauk SSSR Ser. Mat.},
      volume={53},
      number={6},
       pages={1183\ndash 1205, 1337},
      review={\MR{1039961 (91b:14013)}},
}

\bib{BOr}{inproceedings}{
      author={Bondal, A.},
      author={Orlov, D.},
       title={Derived categories of coherent sheaves},
        date={2002},
   booktitle={Proceedings of the {I}nternational {C}ongress of
  {M}athematicians, {V}ol. {II} ({B}eijing, 2002)},
   publisher={Higher Ed. Press, Beijing},
       pages={47\ndash 56},
      review={\MR{1957019 (2003m:18015)}},
}

\bib{Bo}{article}{
      author={Bondal, A.~I.},
       title={Derived categories of toric varieties},
        date={2006},
     journal={Convex and Algebraic geometry, Oberwolfach conference reports},
      volume={3},
       pages={284\ndash 286},
}

\bib{FLTZ}{article}{
      author={Fang, Bohan},
      author={Liu, Chiu-Chu~Melissa},
      author={Treumann, David},
      author={Zaslow, Eric},
       title={A categorification of {M}orelli's theorem},
        date={2011},
        ISSN={0020-9910},
     journal={Invent. Math.},
      volume={186},
      number={1},
       pages={79\ndash 114},
         url={http://dx.doi.org/10.1007/s00222-011-0315-x},
      review={\MR{2836052 (2012h:14036)}},
}

\bib{FLTZ2}{article}{
      author={Fang, Bohan},
      author={Liu, Chiu-Chu~Melissa},
      author={Treumann, David},
      author={Zaslow, Eric},
       title={T-duality and homological mirror symmetry for toric varieties},
        date={2012},
        ISSN={0001-8708},
     journal={Adv. Math.},
      volume={229},
      number={3},
       pages={1875\ndash 1911},
         url={http://dx.doi.org/10.1016/j.aim.2011.10.022},
      review={\MR{2871160 (2012m:14075)}},
}

\bib{KS}{book}{
      author={Kashiwara, Masaki},
      author={Schapira, Pierre},
       title={Sheaves on manifolds},
      series={Grundlehren der Mathematischen Wissenschaften [Fundamental
  Principles of Mathematical Sciences]},
   publisher={Springer-Verlag, Berlin},
        date={1990},
      volume={292},
        ISBN={3-540-51861-4},
         url={http://dx.doi.org/10.1007/978-3-662-02661-8},
        note={With a chapter in French by Christian Houzel},
      review={\MR{1074006 (92a:58132)}},
}

\bib{Nad}{article}{
      author={Nadler, David},
       title={Microlocal branes are constructible sheaves},
        date={2009},
        ISSN={1022-1824},
     journal={Selecta Math. (N.S.)},
      volume={15},
      number={4},
       pages={563\ndash 619},
         url={http://dx.doi.org/10.1007/s00029-009-0008-0},
      review={\MR{2565051 (2010m:53131)}},
}

\bib{NZ}{article}{
      author={Nadler, David},
      author={Zaslow, Eric},
       title={Constructible sheaves and the {F}ukaya category},
        date={2009},
        ISSN={0894-0347},
     journal={J. Amer. Math. Soc.},
      volume={22},
      number={1},
       pages={233\ndash 286},
         url={http://dx.doi.org/10.1090/S0894-0347-08-00612-7},
      review={\MR{2449059 (2010a:53186)}},
}

\bib{Or}{article}{
      author={Orlov, D.~O.},
       title={Projective bundles, monoidal transformations, and derived
  categories of coherent sheaves},
        date={1992},
        ISSN={0373-2436},
     journal={Izv. Ross. Akad. Nauk Ser. Mat.},
      volume={56},
      number={4},
       pages={852\ndash 862},
         url={http://dx.doi.org/10.1070/IM1993v041n01ABEH002182},
      review={\MR{1208153 (94e:14024)}},
}

\bib{SS}{article}{
      author={Scherotzke, Sarah},
      author={Sibilla, Nicolo},
       title={The non-equivariant coherent-constructible correspondence and a
  conjecture of king},
        date={2016},
     journal={Selecta Math. (N.S.)},
      volume={22},
      number={1},
       pages={389\ndash 416},
}

\bib{Tr}{article}{
      author={Treumann, D.},
       title={Remarks on the nonequivariant coherent-constructible
  correspondence for toric varieties},
        date={2010},
     journal={eprint arXiv:1006.5756},
}

\end{biblist}
\end{bibdiv}
\end{document}